\pdfoutput=1
\documentclass[11pt]{article}
\usepackage{amssymb}
\usepackage{graphicx}
\graphicspath{{Images/}}
\usepackage{xcolor} 
\usepackage{tensor}
\usepackage{fullpage} 
\usepackage{amsmath}
\usepackage{amsthm}
\usepackage{verbatim}
\usepackage{mathrsfs}
\usepackage{soul}

\usepackage{scalerel,stackengine}
\stackMath
\newcommand\wwhat[1]{%
\savestack{\tmpbox}{\stretchto{%
  \scaleto{%
    \scalerel*[\widthof{\ensuremath{#1}}]{\kern-.6pt\bigwedge\kern-.6pt}%
    {\rule[-\textheight/2]{1ex}{\textheight}}
  }{\textheight}%
}{0.5ex}}%
\stackon[1pt]{#1}{\tmpbox}%
}
\usepackage{enumitem}
\setlist[enumerate]{leftmargin=1.5em}
\setlist[itemize]{leftmargin=1.5em}

\usepackage{seqsplit}

\definecolor{green}{rgb}{0,0.8,0} 

\newtheorem{theorem}{Theorem}[section]
\newtheorem{corollary}[theorem]{Corollary}
\newtheorem{lemma}[theorem]{Lemma}
\newtheorem{proposition}[theorem]{Proposition}

\theoremstyle{definition}
\newtheorem{definition}[theorem]{Definition}

\theoremstyle{remark}
\newtheorem{remark}[theorem]{Remark}

\numberwithin{equation}{section}
\newcommand{\RN}[1]{%
  \textup{\uppercase\expandafter{\romannumeral#1}}%
}

\newcommand{\nnrm}[1]{{\vert\kern-0.25ex\vert\kern-0.25ex\vert #1 
    \vert\kern-0.25ex\vert\kern-0.25ex\vert}}

\newcommand{\supp}{{\mathrm{supp}}\,}

\newcommand{\ud}{\mathrm{d}}

\newcommand{\nb}{\nabla}

\newcommand{\tht}{\theta}

\newcommand{\bbC}{\mathbb C}

\newcommand{\bbN}{\mathbb N}

\newcommand{\bbR}{\mathbb R}

\newcommand{\bbT}{\mathbb T}

\begin{document}

\title{Non-convergence of the rotating stratified flows toward the  quasi-geostrophic dynamics}
\author{Min Jun Jo\thanks{Mathematics Department, Duke University, Durham NC 27708 USA. E-mail: minjun.jo@duke.edu} \and Junha Kim\thanks{Department of Mathematics, Ajou University, Suwon 16499, Republic of Korea. E-mail: junha02@ajou.ac.kr} \and Jihoon Lee\thanks{Department of Mathematics, Chung-Ang University, Seoul 06974, Republic of Korea. E-mail: jhleepde@cau.ac.kr}}
\date{\today}



\maketitle


\begin{abstract}
The quasi-geostrohpic (QG) equation has been used to capture the asymptotic dynamics of the rotating stratified Boussinesq flows in the regime of strong stratification and rapid rotation. In this paper, we establish the invalidity of such approximation when the rotation-stratification ratio is either fixed to be unity or tends to unity sufficiently slowly in the asymptotic regime: the difference between the rotating stratified Boussinesq flow and the corresponding QG flow remains strictly away from zero, independently of the intensities of rotation and stratification. In contrast, we also show that the convergence occurs when the rotation-stratification ratio is fixed to be a number other than unity or converges to unity sufficiently fast. As a corollary, we compute a lower bound of the convergence rate, which blows up as the rotation-stratification ratio goes to unity.
 
\end{abstract}

\section{Introduction}
\subsection{Quasi-geostrophy of rotating stratified fluids}\label{sec2}
We study the rotating stratified Boussinesq equations
\begin{equation}\label{eq_main}
   \left\{\begin{aligned}
\partial_t v  + ({v} \cdot \nabla){v} &= -\Omega e_3 \times {v} -\nabla p + N\theta e_3, \\
\partial_t\theta + ({v} \cdot \nabla)\theta &= -N v_3, \\
\nabla \cdot {v }&= 0,
\end{aligned}
\right. 
\end{equation}
 in $\bbR^3\times(0,\infty)$ with the corresponding initial data $v(x,0)=v_0(x)$ and $\theta(x,0)=\theta_0(x)$. The unknowns $v$, $\theta$, and $p$ represent the fluid velocity, the density disturbance, and the pressure, respectively. The constant $\Omega$ stands for the rotation frequency, while the other constant $N$ denotes the intensity of stratification. Such a system \eqref{eq_main} can be derived by looking at the solution as the perturbation around the stratified (linearly and stably) densities in the original inviscid Boussinesq equations on the rotating reference frame, with suitable normalization.

 The simultaneous presence of both rotation and stratification is known to produce highly complicated oscillations that are difficult to analyze, hampering precise description of the dynamics. To resolve this issue, Charney \cite{Charney48} in the 1940's considered the \emph{quasi-geostrophic} (QG) equations
 \begin{equation}\label{QG}
\left\{
\begin{alignedat}{3}
\partial_t q + &\mathbf{v}_{H}\cdot \nabla_H q=0,  \\
q &= \Delta_\mu \psi, \\
\mathbf{v}_{H}&=(-\partial_2\psi,\partial_1 \psi)^{\top}, \\
\end{alignedat}
\right.
\end{equation}
as the legitimate approximate equations for large-scale atmospheric motions near the geostrophy meaning that the pressure is almost balanced with the Coriolis effect. Note that quasi-hydrostatic dynamics, which is related to stratification of our interest, has been investigated as well. See \cite{Charney48} for details. We also refer to \cite{Charney} for turbulence theory of the QG system.

In the above formulation, we use the following notations
\begin{equation}\label{QG_def_aux}
     \nabla_{\mu} := (-\partial_2, \partial_1, 0, \mu \partial_3)^{\top}, \quad \Delta_\mu := \partial_1^2+\partial_2^2+\mu^2\partial_3^2, \quad  \mbox{and} \quad \nabla_{H}^{\perp}=(-\partial_2,\partial_1)^{\top},
 \end{equation}
where $\mu\in \bbR $ denotes the rotation-stratification ratio $$\mu=\frac{\Omega}{N}.$$ Moreover, the QG system \eqref{QG} may be supplemented with the initial data of the form
\begin{equation}\label{QG_def_initial}
\nabla_\mu \psi (x,0)= \nabla_\mu \psi_0 (x) = P_\mu u_0,
\end{equation}which even  alludes to the relationship between the Boussinesq flow \eqref{eq_main} and the QG flow \eqref{QG} because $u_0$ will be defined by $u_0=(v_0,\theta_0)$  for \eqref{eq_main} and $P_\mu$ will be defined in \eqref{def_proj} as a part of eigenvalue analysis for the linear propagator in \eqref{eq_main}. The main unknown quantity $q$ in \eqref{QG} is called the pseudo-potential vorticity in the geophysics literature because it is conserved only at the horizontal projection of the particle motion \cite{Charney}.  

Such a simple reduced model \eqref{QG} is known to preserve the key features of the real dynamics. One can notice that there is no vertical advection in \eqref{QG}; essentially, the dynamics is quasi-2D. So a natural candidate for the turbulence of \eqref{QG} is the inverse cascade, which appears to be more generic in 2D than the forward cascade does, see \cite{Alex} for a recent numerical result on such inverse cascade in the 3D stratified rotating fluids. 

\subsubsection{Goal of this paper}
The purpose of this article is to provide a mathematically rigorous validation or invalidation of \eqref{QG} as the legitimate approximate description of large-scale geophysical flows whose motions are governed by the 3D 
 inviscid rotating stratified Boussinesq equations \eqref{eq_main}, in the asymptotic regime where both $\Omega$ and $N$ tend to infinity.  In the later sections, we would set up the problem rigorously with precise definitions.

 \subsection{Summary}
 We summarize the contributions of our main results that will be presented in the next chapter.
 \begin{itemize}
     \item \textbf{Specification of $\mu$ for non-convergence.} We fully specify the set of the rotation-stratification ratios where non-convergence actually takes its place; $\mu=1$ is the only element of that set in the case of the whole space $\bbR^3$. This is proved by the combination of our first main theorem, Theorem~\ref{thm3}, and Theorem~\ref{thm1} in Appendix.
     \item \textbf{Pseudo-dichotomy when $\mu\to 1$.} When $\mu$ goes to $1$,  we prove that (i) if $\Omega$ and $N$ grow sufficiently slow, then non-convergence happens, and that (ii) if $\Omega$ and $N$ grow sufficiently fast, still convergence is guaranteed.  This suggests that not only the ratio $\mu=\Omega/N$ but also the constants $\Omega$ and $N$ play the decisive roles in the convergence process. More importantly, our non-convergence result for the variant case $\mu\to 1$ indicates that the QG approximation is indeed \emph{singular,} yielding the first result of this kind compared to the previous related studies \cite{MS,Tak1} that provide the convergence rates.

     \item \textbf{Estimation of convergence rates.} We provide a new quantitative lower bound of the convergence rates for the QG approximation.  Corollary~\ref{cor_div} is the very first \emph{blow-up} result of the harmonic analytical constant $C_\mu$ as $\mu\to 1$ (see the constant $C_\mu$ appearing in Theorem~\ref{thm1}), which follows from our pseudo-dichotomy result of Theorem~\ref{dich}. More precisely, while the harmonic analytical nature of the obtained convergence rate itself  does not provide an explicitly computable lower bound in general, we can leverage Theorem~\ref{dich} to construct a contradictory sequence of pairs $(\Omega,N)$ whenever the rate does not satisfy the lower bound we suggest in Corollary~\ref{cor_div}. 

 \end{itemize}

\subsection{Background}

\subsubsection{Three-scale singular limit}

It is well-established that both rotation and stratification have similar stabilizing effects that work on turbulent fluids via making the dynamics of the fluids quasi-two dimensional; we refer to \cite{BMN97, BMN99, Baroud, Charve05, CDGG02, CDGG06, Gill, Hough, IT13, IT15, KLT14, KLT, Pedlosky, P16, SA18, T17,  Analogy} for the previous various 
results. It is important to consider rotation and stratification together in action, because their simultaneous presence gives rise to interesting phenomena including jets, fronts, and merger of vortices, that might have been easily suppressed if there is only one strong effect, either rotation or stratification. As an example, statistically the most ubiquitous form of vortices is found to have the aspect ratio of approximately $0.8\Omega/N$ \cite{shape} where $\Omega$ and $N$ represent the intensities of rotation and stratification, see \eqref{eq_main}.

However, when both effects take place with significantly strong intensities,
it is still vague how the two different effects in action would affect the fluid dynamics in a rigorous fashion. This leads to the necessity of the simple reduced model equations that would describe the crucial features of the original equations. The QG equations \eqref{QG} have been proposed to be such a reduced model system. To rigorously justify the utility of the QG equations as the legitimate approximation of the Boussinesq equations \eqref{eq_main} for rotating stratified fluids, we prove the convergence of \eqref{eq_main} to \eqref{QG} in the regime of the so-called \emph{three-scale singular limit} problem incorporating the three limits
 \begin{equation}\label{def_Burger}
     \Omega\to\infty, \quad N\to\infty, \quad \mu:=\frac{\Omega}{N} \to \nu
 \end{equation}
for some fixed $\nu\in(0,\infty).$ The category into which the convergence problem falls is named \emph{singular limit} problem because oscillations from the large parameters $\Omega$ and $N$ might prevent the convergence of the solutions as the parameters tend to infinity. Note that $\mu\to\nu$ is included as a part of the three-scale singular limits; the case $\nu=1$ actually proves to be singular in some sense, see the second part \eqref{invalid} of Theorem~\ref{dich}. As $\nu$ varies in \eqref{def_Burger}, the corresponding asymptotic regime changes. The endpoint case $\nu=0$ was handled by Mu and Schochet in \cite{MS}, where it is shown that convergence rates of solutions can be proved via suitable Strichartz estimates. Earlier than \cite{MS}, Takada \cite{Tak1} treated the non-trivial one-scale singular limit problem  $\Omega=0$, $N\to\infty$, and  $\mu=0=\nu$, which may be regarded as a special case of the endpoint case $\nu=0$. It is notable that while the previous convergence rate results \cite{rate2, rate1} for three-scale limits involved certain well-prepared initial data only, such results were obtained in \cite{MS,Tak1} for general Sobolev initial data without the well-preparedness condition. The results of this paper are also proved for general Sobolev initial data. See also \cite{general1,general2} for the broader theory of three-scale singular limits. For the other endpoint case $\nu=\infty$ where rotation is dominant, see the very recent result \cite{MW} by Mu and Wei.

\subsubsection{Previous works for the QG system}

 Embid and Majda (1996) in \cite{EM} gave the first mathematically rigorous proof of the \emph{weak} convergence of the equations \eqref{UEQ} to the QG  equations \eqref{QG} on a fixed time interval $[0,T]$ as $\varepsilon\to0$ for $\varepsilon^{-1}$ being proportional to both $N$ and $\Omega$, where $T>0$ is \emph{sufficiently small} depending on the initial data only. The result was valid even for initial data that were \emph{not} in hydrostatic nor geostrophic balance, while the results \cite{BB,Chemin} relied on the strategy developed by Klainerman and Majda \cite{KM} for singular limits of hyperbolic problems requiring initial data to be balanced. The balance assumption is natural in the sense that we are looking for perturbative solutions around the stable states of geostrophic balance. The term `quasi-geostrophic' means that the solutions are in \emph{almost} geostrophic balance.

Then Babin, Mahalov, Nicolaenko, and Zhou (1997) presented an improved result in \cite{BMNZ}, which states that the convergence time $T$ can be made arbitrarily large. According to \cite{BMNZ}, while Embid and Majda missed the regularizing effects of rotation and stratification (the time interval constructed in \cite{EM} was essentially as small as the one for the local existence of the 3D Euler equations), a detailed Fourier analysis of \cite{BMNZ} on the torus $\bbT^3$ lead to the claim that the maximal convergence time can be extended to any positive number for certain Sobolev initial data with Burger number outside a Lebesgue measure zero set in $\bbR_{+}$.
 The regularization effects of rotation and stratification turned out to be well-captured by dispersive estimates for the linear propagator $e^{iNtp_{\mu}(D)}$. For examples, see the studies \cite{CDGG06,IT13,IT15,KLT14,KLT,LT17} that are based on such dispersive estimates. This paper particularly concerns the regime $\operatorname{Bu}=O(1)$, or equivalently $0<\mu<\infty$ converging to $\nu\in(0,\infty)$ as $N\to \infty$ and $\Omega\to\infty$. Here, the Burger number $\operatorname{Bu}$ is defined by $\operatorname{Bu}=\frac{N^2a_3^2}{4\Omega^2}$ where $a_3$ represents the aspect ratio parameter that is often captured by considering the domain $\bbT^2 \times [0,2\pi/a_3]$.

 One interesting remark made without proof in \cite{BMNZ}  was that certain nonlinear terms appearing in their decomposition of the limit equations turned out to be \emph{discontinuous} in the Burger number $\operatorname{Bu}$ (note that the Burger number is a quantity that represents the rotation-stratification ratio normalized by the horizontal and the vertical scales of the periodic boxes) at every $\operatorname{Bu}$ belonging to a countable set, called the strict three-wave resonant set. This implies, according to \cite{BMNZ}, that solutions of the limit system would \emph{discontinuously} depend on the Burger number $\operatorname{Bu}$ as well. The authors of \cite{BMNZ} further claimed that since the solutions of the original inviscid Boussinesq equations \eqref{eq_main} depend on $\operatorname{Bu}$ continuously on a small time interval, the convergence to the solutions of the limit equations \emph{could not be made uniform in $\operatorname{Bu}$.} The authors of \cite{BMNZ} named this paradox ``Devil's staircase of convergence results." Since our domain of interest in this paper is simply $\bbR^3,$ such resonant behavior in the periodic boxes cannot be captured anymore and so there is no connection between our result and the paradox. Despite such a difference, our result of non-convergence is inspired by the conjecture.


 Following the  notations of the recent result \cite{Tak1} that dealt with stratification only, here we prove the pseudo-dichotomy on the rotation-stratification ratios to determine either convergence or non-convergence in the presence of both rotation and stratification with strong intensities. Note that non-convergence at some fixed $\mu$ would immediately imply non-uniformity of the convergence rates near such $\mu$.

\section {Main results}
\subsection{Non-convergence for $\mu=1$}

We prove in Appendix that $\mu\neq1$ implies convergence. We can even show that as long as $\mu$ converges to a number other than one, convergence is guaranteed; See Corollary~\ref{cor_conv} and the corresponding proof. Only the case $\mu=1$ is left, i.e., $\Omega=N$. In that case, we establish the following non-convergence theorem. To state our first main theorem, we introduce the projection operator $P_\mu$ defined by 
\begin{equation}\label{Pmu}
    P_\mu w = \int e^{2\pi i x \cdot \xi}\langle \mathscr{F} w(\xi),b_\mu(\xi)\rangle_{\bbC^{4}} b_\mu(\xi) \, \mathrm{d}\xi \quad \mbox{with}\quad b_\mu(\xi) = \frac {\xi_\mu}{|\xi_\mu|},
\end{equation}
where $\xi_{\mu}=(-\xi_2,\xi_1,0,\mu \xi_3)^{\top}.$ We refer to Section~\ref{eigen} for further information.

\begin{theorem}\label{thm3}
 Fix any $m\in \mathbb{N}$ with $m\geq 6$. For any $u_0=(v_0,\theta_0)\in H^m(\bbR^3)$ with $\nb \cdot v_0=0$ and $u_0-P_1 u_0 \neq 0$ where $P_{1}$ is a projection operator defined in \eqref{Pmu} with $\mu=1$, there exists a constant $A>0$ depending only on $u_0$ such that the unique classical local-in-time solution $u=(v,\theta)$ to \eqref{eq_main} with $\Omega=N$ satisfies
 \begin{equation}\label{div}
\liminf_{N\to\infty}\left(\inf_{t \in [0,t_0]} \| u - \nabla_{1} \psi^{1} \|_{W^{1,\infty}}\right) \geq \frac 12 A \end{equation}
 for some sufficiently small $t_0\in(0,T_{\ast}),$ where $T_{\ast}$ is the maximal existence time of $u$ with
 $$u \in C([0,T_\ast);H^m(\mathbb{R}^3)) \cap C^1([0,T_\ast);H^{m-1}(\mathbb{R}^3))$$ and $\psi^1$ is a continuous function which satisfies \eqref{QG}-\eqref{QG_def_initial} when $\mu = 1$ and $$\nabla_1\psi^1 \in C([0,T]);H^m(\mathbb{R}^3)) \cap C^1([0,T];H^{m-1}(\mathbb{R}^3))$$ for all $T>0$.
\end{theorem}
\begin{remark}
    The maximal time $T_{\ast}$ of the local-in-time solution does not depend on $N$.
\end{remark}
\begin{remark} From the simple inequality
\begin{equation*}
    \|f\|_{L^q(0,T;W^{1,\infty})} \geq t_0^{\frac{1}{q}}\inf_{t\in[0,t_0]}\|f\|_{W^{1,\infty}}
\end{equation*}
for any nice vector-valued function $f:\bbR^3\times\bbR_{+}\to\bbR^3$, we can see that the non-convergence result we proved here is actually stronger than the non-convergence in the Strichartz space $L^q(0,T;W^{1,\infty})$ that was previously considered for the convergence results.
\end{remark}

\subsection{Fast convergence vs. slow non-convergence for $\mu\to 1$}
When $\Omega$ and $N$ grow, it is natural to expect that the ratio $\mu=\Omega/N$ changes accordingly. While non-convergence always happens for the fixed ratio $\mu=1$, it is still not clear what will occur in the variant case $\mu\to 1$. For the variant case $\mu\to\nu \neq 1$, see Corollary~\ref{cor_conv}. Surprisingly, it turns out that both convergence and non-convergence can happen separately, being contingent upon how fast $\Omega$ and $N$ grow to infinity. This suggests that not only the rotation-stratification ratio $\mu$ matters, the intensities of rotation and stratification themselves play the crucial roles which distinguish the three-scale singular limit problem from the previous one or two-scale singular limit results.

Let $\{\mu_k\}_{k\in\bbN}$ be any sequence of positive numbers satisfying $\mu_k\to 1$ as $k\to \infty$. Without loss of generality, we may assume $\mu_k\neq 1$ for every $k\in\bbN$. For each fixed $\mu_k$ with $k\in \bbN$, the constant $C_{\mu_k}$ appearing in \eqref{conv_1} is then well-defined.

\begin{definition}
We say that a sequence $\{(\Omega_k,N_k)\}_{k\in\bbN}$ of pairs of positive numbers is \emph{fast} if the sequence satisfies
\begin{equation}\label{fast}
    N_k=C_{\mu_k}h(k), \quad\mbox{and}\quad \Omega_k= N_k \mu_k
\end{equation}
for some function $h\in C(\bbR_{+})$ satisfying $h(x)\to\infty$ as $x\to \infty$. The map $x\mapsto x^\beta$ with $\beta>0$ can be chosen to be such $h$.
\end{definition}
\begin{definition}
  We say that the sequence is \emph{slow} if the sequence satisfies
\begin{equation}\label{slow}
    N_k=o(|1-\mu_k|^{-1}) \,\,\,\mbox{as}\,\,\,k\to \infty, \quad\mbox{and}\quad \Omega_k= N_k \mu_k.
\end{equation} One can see that $N_k=\frac{1}{|1-\mu_k|^{\beta}}$ with $0<\beta<1$ would be an example of \eqref{slow}.
  \end{definition}

Now we state our pseudo-dichotomy result between fast and slow sequences as follows.
\begin{theorem}\label{dich}
 Fix $u_0=(v_0,\theta_0)\in H^m$ with $\nb\cdot v_0=0$ for some $m\geq 7$. Let $\{(\Omega_k,N_k)\}_{k\in\bbN}$ be any sequence of pairs of positive numbers. For every $T>0$, and for any $k\in\bbN$, suppose that there exists a corresponding unique classical solution $u^{(k)}$ to \eqref{eq_main} on $[0,T]$. If the fast condition \eqref{fast} is satisfied, then there holds
\begin{equation}\label{valid}
    \lim_{k\to\infty}\|u^{(k)}-\nabla_1 \psi^{1}\|_{L^q(0,T;W^{1,\infty})} = 0.
\end{equation}
In contrast, if the slow condition \eqref{slow} is satisfied and $u_0 \not= P_1u_0$, then there holds
\begin{equation}\label{invalid}
    \liminf_{k\to\infty}\|u^{(k)}-\nabla_1 \psi^{1}\|_{L^q(0,T;W^{1,\infty})} \geq B,
\end{equation}
for some positive constant $B$ depending only on $u_0.$ Here $\psi^1$ is a continuous function which satisfies \eqref{QG}-\eqref{QG_def_initial} when $\mu = 1$ and $$\nabla_1\psi^1 \in C([0,T]);H^m(\mathbb{R}^3)) \cap C^1([0,T];H^{m-1}(\mathbb{R}^3))$$ for all $T>0$.
\end{theorem}

Lastly, we give a new explicit lower bound for the growth $C_\mu$ in \eqref{conv_1} as $\mu$ tends to $1$, leading to the first blow-up result of $C_\mu$ to the best of the authors' knowledge. This provides a quantitative information on the convergence process. To this end, we exploit both \eqref{valid} and \eqref{invalid}.
\begin{corollary}\label{cor_div}
Let $\{\mu_k\}_{k\in\bbN}$ be any sequence of positive numbers satisfying $\mu_k\to 1$ as $k\to\infty$. For the constant $C_\mu$ appearing in \eqref{conv_1}, there holds
\begin{equation*}
    \frac{1}{C_{\mu_k}} = O\left(|1-\mu_k|\right) \quad \mbox{as}\quad k\to \infty.
\end{equation*}
\end{corollary}
\begin{remark} The above corollary says that $C_\mu$, as a function in $\mu$, increases to infinity as $\mu\to 1$ at least as fast as the mapping $\mu\mapsto\frac{1}{|1-\mu|}$.
\end{remark}

The first part \eqref{valid} of the above theorem supports the validity of the QG equations as the approximation for rotating stratified fluids even near the exact balance between rotation and stratification, provided that we consider an asymptotic regime of radically increasing intensities. Meanwhile, the second part \eqref{invalid} says that such quasi-geostrophic approximation might be mathematically invalid near the exact balance in the regime of slowly growing intensities. Together with Corollary~\ref{cor_conv}, the above theorem covers the whole range $\mu\to\nu\in(0,\infty)$ for the rotation-stratification ratio. The open question that seems currently out of reach regards the \emph{moderately increasing} sequences that belong to neither \eqref{fast} nor \eqref{slow}, when $\mu$ tends to one.

\subsection{$\mu$-universal non-convergence in $L_t^{\infty} H_x^{s}$}

The Sobolev space $W^{1,\infty}$, which is $L^\infty$-based, is nearly \emph{marginal} in the following sense: while $W^{1,\infty}$ gives us the convergence results by well-capturing the dispersive nature of the linear operator for any fixed $\mu\neq 1$, it still allows for non-convergence when $\mu=1$ is fixed. Roughly saying, the space $W^{1,\infty}$ successfully distinguishes the isotropic case $\mu=1$ from the anisotropic case $\mu\neq 1$.

In contrast, such a dichotomy for the fixed rotation-stratification ratio $\mu$ ceases to work in the smaller $L^2$-based space $H^{m-3}$. One can modify the proof of Theorem~\ref{thm3} to show that non-convergence arises universally in $H^{m-3}$ for any $\mu\in(0,\infty)$ on a sufficiently short time interval. See Theorem~\ref{rmk_div} for the precise statement. The universal non-convergence result can be interpreted as that even the slightly smaller $L^2$-based space $H^{m-3}$ might be inappropriate to study the ``phase scrambling" of the solutions via rotation and stratification. 
\begin{theorem}[A variant of the result in \cite{Grenier}]\label{rmk_div}
    Let $\mu \in (0,\infty)$ and $m \in \mathbb{N}$ with $m \geq 6$. For any $u_0=(v_0,\theta_0) \in H^m(\mathbb{R}^3)$ with $\nabla \cdot v_0 = 0$ and $u_0 - P_{\mu}u_0 \neq 0$, the unique classical local-in-time solution $u=(v,\theta)$ to \eqref{eq_main} satisfies
     \begin{equation*}
        \liminf_{N\to\infty}\left(\inf_{t \in [0,t_0]} \| u - \nabla_{\mu} \psi^{\mu} \|_{H^{m-3}}\right) \geq \frac 12 \| u_0 - P_{\mu} u_0 \|_{H^{m-3}}
    \end{equation*}
    for sufficiently small $t_0\in(0,T_\ast)$ where $T_{\ast}$ is the maximal existence time of $u$ satisfying
 $$u \in C([0,T_\ast);H^m(\mathbb{R}^3)) \cap C^1([0,T_\ast);H^{m-1}(\mathbb{R}^3)).$$ Here $\psi^\mu$ is a continuous function which satisfies \eqref{QG}-\eqref{QG_def_initial} and $$\nabla_\mu\psi^\mu \in C([0,T]);H^m(\mathbb{R}^3)) \cap C^1([0,T];H^{m-1}(\mathbb{R}^3))$$ for all $T>0$.
\end{theorem}
\begin{remark}The theorem indicates that the nonzero status of the initial difference $u_0-P_\mu u_0$ is sustained for a short amount of time not depending on $N$ with respect to the $H^s$ norms.
\end{remark}
\begin{remark}
    The general result in \cite{Grenier} on $\bbR^3$ can be applied to our setting when we work with $H^s$ norms; Theorem 3.5 of \cite{Grenier} establishes convergence of the slow limit part of solution to the target profile in $H^s$ norms that are weaker than the one for which uniform bounds hold. The implication is that if the fast oscillating part is initially non-zero then at least on some short time interval the discrepancy between the full solution and the target profile is bounded from below on such times. So the above non-convergence theorem is essentially a consequence of \cite{Grenier}. In this paper, we provide an elementary proof that is consistent with our approach.
\end{remark}

\subsection{$\mu$-universal convergence in $H^s$}
Despite such incapability of $L^2$-based spaces in measuring the effects of $\Omega$ and $N$ on dispersion, we still can specify a special class for initial data which ensures the convergence for any $\mu\in(0,\infty)$ even in $H^{m-3}$. By imposing a smallness condition on certain components of the initial data in terms of the eigenvectors of the linear propagator, we establish the following universal convergence result.

\begin{theorem}[A variant of the result in \cite{KM}]\label{thm2}
Let $\mu \in (0,\infty)$ and $m\in \mathbb{N}$ with $m\geq 6$. Fix any time $T > 0$ and any initial data $u_0=(v_0,\theta_0) \in H^m(\mathbb{R}^3)$ with $\nabla \cdot v_0 = 0$. If $u_0$ satisfies $CTe^{CT} \| u_0 - P_\mu u_0 \|_{H^{m-2}} \leq 1$ for some $C>0$, then there exists a constant $R=R(m,\mu,T,\|P_{\mu}u_0\|_{H^m})>0$ such that \eqref{UEQ} possesses a unique classical solution 
$$u \in C([0,T];H^m(\mathbb{R}^3)) \cap C^1([0,T];H^{m-1}(\mathbb{R}^3))$$
provided that $\sqrt{\Omega^2+N^2}>R$. Moreover, if we assume further that $u_0 = P_{\mu} u_0$, there exists a constant $C_\mu'=C_\mu'(m,\mu,T,\| u_0 \|_{H^m}) > 0$  such that
\begin{equation}\label{conv_2}
\sup_{t \in [0,T]} \| u - \nb_\mu \psi^\mu \|_{H^{m-3}} \leq \frac {C_\mu'}{\sqrt{\Omega^2+N^2}}
\end{equation}
as long as $\sqrt{\Omega^2+N^2} > R$. Here $\psi^\mu$ is a continuous function which satisfies \eqref{QG}-\eqref{QG_def_initial} and $$\nabla_\mu\psi^\mu \in C([0,T]);H^m(\mathbb{R}^3)) \cap C^1([0,T];H^{m-1}(\mathbb{R}^3))$$ for all $T>0$.
\end{theorem}

\begin{remark}
    This includes $\mu=1$ with $C_\mu'$ bounded near $\mu= 1$, which corroborates that the non-convergence result in Theorem \ref{thm3} heavily relies on certain components of the initial data, $P_{\pm}u_0$, instead of just $u_0$.
\end{remark}

\begin{remark}
    It should be noted that the results by Klainerman and Majda \cite{KM} are valid in the $H^s$ settings; the Lagrangian-multiplier term $\nabla p$ disappears in the standard $H^s$ energy estimates on $\bbR^3$ due to the divergence-free constraint. Here, we give another proof that is consistent with our approach. See \cite{general1} as well. 
\end{remark}



\subsection{Generalized non-convergence theorem}

 Due to its nature of generalization, this section may require certain preliminary knowledge about the proofs of the previous theorems. One may read the next sections before reading this section.
 
 Our proofs of Theorem~\ref{thm3} and Theorem~\ref{rmk_div} gives rise to a general result as follows. Consider the initial value problem in $\bbR^d\times \bbR_{+}$ for a vector-valued quantity $f$ with $n$-many components as

\begin{equation}\label{eq_gen}
\partial_t {f}  + L^{\varepsilon}f + \mathcal{N}(f)=0 \quad \mbox{and}\quad f|_{t=0}=f_0,
\end{equation}
where $\mathcal{N}$ denotes the nonlinearity of the system and $L^{\varepsilon}$ stands for the linear propagation that depends on the parameter $\varepsilon$. We consider the asymptotic regime when
\begin{equation*}
    \varepsilon\to 0.
\end{equation*}
Note that this is our case indeed, because $\mu=\Omega/N$ is fixed so that $N$ and $\Omega$ grow proportionally to each other in Theorem~\ref{thm3} and Theorem~\ref{rmk_div}. 

 Let the linear propagator $L^{\varepsilon}$ be an operator defined by its Fourier symbol $\widehat{L^{\varepsilon}(\xi)}$ that is a complex-valued $n\times n$ matrix. The matrix $\widehat{L^{\varepsilon}(\xi)}$ has the $n$-many eigenvalues denoted by $\lambda_j(\xi)$'s that are well-defined for a.e. $\xi\in \bbC^4$ and the corresponding eigenvectors of the form $b_j(\xi)$'s. Assume that the set of the eigenvectors, $\{b_j(\xi)\}_{j=1}^{n}$, is an orthonormal basis of $\bbC^n$; then the solution $\omega $ to the linearized system
 \begin{equation*}
     \partial_t \omega + L^{\varepsilon} \omega =0 
 \end{equation*}
with the initial data $\omega_0$, can be rewritten in Fourier variables as,
\begin{equation*}
    \widehat{\omega}(t,\xi) = \sum_{j\in \{1\}\cup \mathcal{J}} e^{t \lambda_j^{\varepsilon}(\xi)} \langle \widehat{\omega_0}(\xi),b_j(\xi)\rangle_{\bbC^n} b_j(\xi),
\end{equation*}
where $\mathcal{J}=\{2,3,\cdots, n\}$ is the set of the indices for the eigenpairs except for $(b_1,\lambda_1)$.
Note that the case $j=1$ will represent the limit slow dynamics, so we further assume $\lambda_1(\xi)$ is $\varepsilon$-independent.

\vspace{0.1in}

\noindent \textit{\textbf{Assumptions.}}
\begin{enumerate}
\item \textit{(Local well-posedness)} There is a spatially defined function space $X$ equipped with the norm $\|\cdot \|_{X}$ such that  there exists a unique solution $f\in C([0,T]; X(\bbR^d))$ to \eqref{eq_gen} where $T$ is the $\varepsilon$-independent existence time of $f$.
    \item \textit{(Decomposition)} The unique solution $f$ can be decomposed as
    \begin{equation*}
        f=f_1+\sum_{j\in \mathcal{J}} f_j + \mathcal{E} \end{equation*}
    in the class $C([0,T]; X(\bbR^d))$, where $|\mathcal{J}| \leq {n-1}$, $\mathcal{E}(0)=0$, and $f_1$ is $\varepsilon$-independent.

    \item \textit{(Initial largeness)} There exists a constant $A>0$ such that  $$\left\|\sum_{j\in\mathcal{J}}e^{t\lambda_j^{\varepsilon}}f_{j}(0) \right\|_{X} \geq A$$ for any $\varepsilon>0$ and any $t\in[0,t_0]$ with some sufficiently small $t_0>0$.

    \item \textit{(Convergence)} There exists a $\varepsilon$-independent constant $t_0 \in (0,T_*)$ such that the difference between $f_j$'s and their corresponding linear propagation,  $$\sum_{j\in \mathcal{J}} \left\| f_j(t)-e^{t\lambda_j^{\varepsilon}}f_j(0) \right\|_{X},$$ converges to $0$ as $\varepsilon\to 0.$
     \item \textit{(Uniform smallness)} The error $\mathcal{E}$ satisfies the smallness condition $$\| \mathcal{E}(t) \|_{X} \leq \frac{1}{4}A$$ over the time interval $[0,t_0].$ 
\end{enumerate}

\begin{theorem}\label{thm_gen}
    If the conditions in \textbf{Assumptions} are true, then the solution $f$ to \eqref{eq_gen} satisfies
    \begin{equation*}
        \liminf_{\varepsilon\to 0}\left(\inf_{t\in[0,t_0]}\|f(t)-f_1(t)\|_{X}\right) \geq \frac{1}{2} A,
    \end{equation*}
    where $A>0$ is a constant depending only on $u_0$ and $ t_0\in(0,T_\ast)$ is sufficiently small.
\end{theorem}

\begin{remark}
    The proof can be directly obtained from the triangle inequality.
    \end{remark}
    \begin{remark} We may recall the nature of the proofs of the previous theorems. In proving our previous theorems, introducing the modified linear system \eqref{MLS} was crucial to satisfy the conditions on both \textit{decomposition} and \textit{convergence}. Note also that, during the proof of Theorem~\ref{thm3}, the uniform-in-$\xi$ form of $\lambda_j^\varepsilon = it\frac{1}{\varepsilon}$ was necessary to satisfy the condition of \textit{initial largeness} due to the nature of $X=W^{1,\infty}$. For Theorem~\ref{rmk_div} where $X=H^s$, one can see that as long as the eigenvalue is purely imaginary, the $H^s$ norm wouldn't be affected by $e^{t\lambda_j^\varepsilon}$ even when $\lambda_j^\varepsilon$ depends on $\xi$, and so the initial largeness is trivially obtained. The smallness of the error term $\mathcal{E}$ near the initial time not depending on $\varepsilon$ was shown by the energy-type estimates.
\end{remark}

Whether such sufficient conditions hold or not depends on the structure of the linear propagator $L^{\varepsilon}$ and the nonlinearity $\mathcal{N}(f)$. Recalling that we denote by $P_1$ the projection onto the space spanned by $b_1(\xi)$ in Fourier variables as the slow limit component, we can specify the required structure as follows.

\begin{proposition} 
    Let $X=W^{k,\infty}$ with $1 \leq k \in \bbN$. Assume that $\lambda_1^{\varepsilon}(\xi) = 0$ and $\lambda_j^{\varepsilon}(\xi)$ takes the form  $\lambda_j^{\varepsilon}(\xi) = ih_j(\varepsilon)$, where the function $h_j:\bbR_{+}\to\bbR_{+}$ satisfies $h_j(\varepsilon)\to \infty$ as $\varepsilon\to  0$ for each $j \in \mathcal{J}$. Fix any $s \in \bbN$ with $s > \frac{d}{2} + k$. Assume that the nonlinearity $\mathcal{N}$ in \eqref{eq_gen} takes the bilinear form and satisfies the estimates
    \begin{equation*}
        \left| \langle \partial^\alpha [\mathcal{N}(g,\bar{g})],\partial^\alpha \bar{g} \rangle_{L^2} \right| \leq C \|g\|_{H^{s}} \| \bar{g} \|_{H^{s}}^{2}, \quad \left\| \partial^\alpha [\mathcal{N}(g,\bar{g})] \right\|_{L^2} \leq C \|g\|_{H^{s}} \| \bar{g} \|_{H^{s+\ell}}
    \end{equation*} for any multi-index $\alpha$ with $|\alpha| < s$ and     \begin{equation*}
        \left| \langle \partial^\alpha [\mathcal{N}(g,\bar{g})],\partial^\alpha \bar{g} \rangle_{L^2} \right| \leq C \|g\|_{H^{a}} \| \bar{g} \|_{H^{a}}^{2}, \quad \left\| \partial^\alpha [\mathcal{N}(g,\bar{g})] \right\|_{L^2} \leq C \|g\|_{H^{a}} \| \bar{g} \|_{H^{a+\ell}}
    \end{equation*}
     for any multi-index $\alpha$ with $|\alpha| = a \geq s$, where $\ell$ is some natural number. Regarding the linearity $L^\varepsilon$, we assume that  $L^\varepsilon$ is a Fourier multiplier whose symbol $\mathscr{F}L^{\varepsilon}(\xi)$ is a skew-symmetric matrix. Then, for any initial data $f_0\in H^m$ with $m\in \bbN$ satisfying $m \geq s + 3\ell$, \textbf{\textit{Assumptions}} can be satisfied by introducing the corresponding modified linear system that is analogous to \eqref{MLS}.
\end{proposition}

\begin{remark}
    The case $X=H^{s}$ with $\lambda_j^{\varepsilon}(\xi) = ih_j(\varepsilon)p_j(\xi)$ where $\frac{1}{C} \leq |p_j(\xi)| \leq C$ can be treated similarly, in an analogous fashion to the proof of Theorem~\ref{rmk_div}.
\end{remark}

\begin{remark}
    Together with Theorem~\ref{thm_gen}, this proposition generalizes our non-convergence results to the equations of the form \eqref{eq_gen} with the required structure stated above. One can prove this proposition by following the proofs of Theorem~\ref{thm3} and Theorem~\ref{rmk_div}.
\end{remark}

\section{Preliminaries}\label{sec_prelim}
\subsection{Local well-posedness}
The Boussinesq system \eqref{eq_main} is locally well-posed in $H^s$ for $s>\frac{5}{2}$. Suppose that $(v_0,\theta_0)\in H^{s}$. For any given $\Omega>0$ and $N>0$, by the standard theory for local well-posedness, there exists a maximal existence time $T_\ast=T_\ast(s,\|(v_0,\theta_0)\|_{H^{s}})$, which is obtained independently of the geophysical parameters $\Omega$ and $N$ thanks to the \emph{skew-symmetric} structure of the linear part. The system \eqref{eq_main} possesses a unique classical local-in-time solution pair $(v,\theta)$ satisfying $$(v,\theta) \in C([0,T_\ast);H^{s}(\bbR^3))\cap C^1([0,T_\ast);H^{s-1}(\bbR^3)).$$ Moreover, there are positive constants $C_0=C_0(s)$ and $C_1=C_1(s,T)$ such that
 \begin{equation}\label{loc_est}
     T_\ast\geq \frac{C_0}{\|(v_0,\theta_0)\|_{H^{s}}},\quad \sup_{t\in[0,T)}\|(v,\theta)(t)\|_{H^{s}}\leq C_1\|(v_0,\theta_0)\|_{H^{s}},
 \end{equation}
 for any $T\in(0,T_\ast)$.

 \subsection{Linear theory}\label{sub3.1}

    Our technical goal is to exploit the linear structure of \eqref{eq_main}, which comes from the three linear terms $-\Omega e_3\times v,$ $N\theta e_3,$ and $-Nv_3$ that are contingent upon the presence of either rotation or stratification. As the parameters $N$ and $\Omega$ get larger, we expect that the aforementioned linear terms will dominate the dynamics of \eqref{eq_main}; the solution to \eqref{eq_main} would either converges to or stays away from the solution of the corresponding limit system, whose structure is determined by the eigenvector of the linear propagator. 
    
    To witness such convergence or non-convergence more tangibly, we investigate the linear structure of the above system by consolidating the separate linear terms into one linear operator. We first incorporate $v$ and $\theta$ into one vector field $u$, say $
u=(v,\theta)$ and then apply the 3D Leray-Helmholtz projector to the velocity equations in \eqref{eq_main} to remove the pressure gradient. This allows us to rewrite the system \eqref{eq_main} as the  following amalgamated version
\begin{equation}\label{UEQ}
\left\{
\begin{alignedat}{3}
u_t + N\,\widetilde{\mathbb{P}} J \widetilde{\mathbb{P}} \,u + \widetilde{\mathbb{P}}\,(u \cdot \widetilde{\nabla})u &= 0, \\ \widetilde{\nabla} \cdot u& = 0, \\  
\end{alignedat}
\right.
\end{equation}
where we have used the following definitions
\begin{gather}\label{proj_notation}
\widetilde{\mathbb{P}} := \left(
\begin{array}{c|c}
\mathbb{P} & 0\\
\hline
0 & 1
\end{array}
\right), \qquad J := \begin{pmatrix}
0 & -\mu & 0 & 0 \\
\mu & 0 & 0 & 0 \\
0 & 0 & 0 & -1 \\
0 & 0 & 1 & 0
\end{pmatrix}, \qquad \widetilde{\nabla} := \begin{pmatrix} \nabla \\ 0 \end{pmatrix}.
\end{gather}
In the above notations, $\mathbb{P}$ is the usual Leray-Helmholtz projector. One may use the simple facts $\mathbb{P}\,\nabla p=0$ and $\widetilde{\mathbb{P}} \, u=u$ to derive the system \eqref{UEQ} from \eqref{eq_main}.

\subsubsection{The Linearized Equations}\label{subsec_LE}
Here, we obtain the explicit solution formula regarding the linearized equations for \eqref{UEQ}, in view of the eigenvectors of the linear operator. Recalling the definitions in \eqref{proj_notation}, let us start with the linearized system
\begin{equation}\label{LEQ}
   \partial_t w + N\widetilde{\mathbb{P}} J \widetilde{\mathbb{P}} w = 0, \qquad \widetilde{\nabla}\cdot w = 0,
\end{equation}
supplemented with the divergence-free initial data $w(0,x)=w_0(x)$. Taking the Fourier transform $\mathscr{F}$, we rewrite \eqref{LEQ} as

\begin{equation}\label{FLEQ}
\partial_t \mathscr{F} w  + N\widetilde{P}(\xi) J \widetilde{P}(\xi) \mathscr{F} w = 0, \qquad  \begin{pmatrix}
\xi \\
0
\end{pmatrix} \cdot \mathscr{F} w = 0,
\end{equation}
equipped with the corresponding initial condition $\mathscr{F} w(0,\xi)=\mathscr{F} w_0(\xi)$ satisfying $\begin{pmatrix}
\xi \\
0
\end{pmatrix} \cdot \mathscr{F} w_0 = 0$ due to the divergence-free condition. The matrix $\widetilde{P}(\xi)$ is the Fourier symbol of the projection $\widetilde{\mathbb{P}}$, which can be expressed as
$$
\widetilde{P}(\xi) = \left(
   \begin{array}{c|c}
     P(\xi) & 0\\
     \hline
     0 & 1
   \end{array}
   \right)
= \frac 1{|\xi|^2}
\begin{pmatrix}
\xi_2^2 + \xi_3^2 & - \xi_1 \xi_2 & - \xi_1 \xi_3 & 0 \\
- \xi_2 \xi_1 & \xi_1^2 + \xi_3^2 & - \xi_2 \xi_3 & 0 \\
- \xi_3 \xi_1 & - \xi_3 \xi_2 & \xi_1^2 + \xi_2^2 & 0 \\
0 & 0 & 0 & |\xi|^2
\end{pmatrix}.
$$
Throughout the paper, we will frequently use the following notations:
\begin{equation}\label{vectors}
\xi_H = (\xi_1,\xi_2,0,0)^{\top} \quad \mbox{and}\quad \xi_{\mu}=(-\xi_2,\xi_1,0,\mu\xi_3)^{\top}.
\end{equation}
The vector $\xi_H$ represents horizontal part of $(\xi,0)\in\bbC^4$ and $\xi_{\mu}$ stands for certain linear transformation of $(\xi,0)$ in terms of $\mu=\frac{\Omega}{N}$.
Setting up the linear operator $L$ by $L(\xi):=-NP(\xi)JP(\xi),$
a direct computation gives the representation for $L$ as
$$
L(\xi) =
\frac 1{|\xi|^2}
\begin{pmatrix}
0 & \Omega\xi_3^2 & -\Omega\xi_2\xi_3 & -N\xi_1\xi_3 \\
-\Omega\xi_3^2 & 0 & \Omega\xi_1\xi_3 & -N\xi_2\xi_3 \\
\Omega\xi_2\xi_3 & -\Omega\xi_1\xi_3 & 0 & N|\xi_H|^2 \\
N\xi_1\xi_3 & N\xi_2\xi_3 & -N|\xi_H|^2 & 0
\end{pmatrix}.
$$

\subsubsection{Eigenvalues and Eigenvectors of the Linear Operator $L(\xi)$}\label{eigen}
It has been proved in \cite{IMT} that the matrix $L(\xi)$ has four eigenvalues with their explicit forms $\{\lambda_\gamma\}_{\gamma=\mathcal{N},\mu,\pm}=\{0,0,iN\frac{|\xi_{\mu}|}{|\xi|},-iN\frac{|\xi_{\mu}|}{|\xi|}\}$ and the corresponding eigenvectors $\{b_\gamma(\xi)\}_{\gamma=\mathcal{N},\mu,\pm}$ satisfying
\begin{equation}\label{def_b}
   \begin{split}
      b_\mathcal{N}(\xi)&=\frac{1}{|\xi|}\begin{pmatrix}
       \xi_1 \\ \xi_2 \\ \xi_3 \\ 0
       \end{pmatrix},
       \qquad
       b_\mu(\xi)=\frac{\xi_\mu}{|\xi_\mu|}, \qquad
       b_{\pm}(\xi)=\frac{1}{\sqrt{2}|\xi_H||\xi||\xi_{\mu}|}
       \begin{pmatrix}
       \mu\xi_2\xi_3|\xi|\pm i\xi_1\xi_3|\xi_{\mu}| \\
       -\mu\xi_1\xi_3|\xi|\pm i\xi_2\xi_3|\xi_{\mu}| \\
       \mp i|\xi_H|^2|\xi_{\mu}| \\
       |\xi_H|^2|\xi|
       \end{pmatrix}.
   \end{split}
\end{equation}
Indeed, $\{b_\gamma(\xi)\}_{\gamma=\mathcal{N},\mu,\pm}$ is an orthonormal basis of $\mathbb{C}^4$ which satisfies
\begin{equation*}
\begin{split}
   L(\xi)b_{\mu}(\xi)=L(\xi)b_{\mathcal{N}}(\xi)=0, \quad L(\xi)b_{+}(\xi)=iN\frac{|\xi_{\mu}|}{|\xi|}b_{+}(\xi),\quad L(\xi)b_{-}(\xi)=-iN\frac{|\xi_{\mu}|}{|\xi|}b_{-}(\xi).
\end{split}
\end{equation*}
Since we have $\langle\mathscr{F} w_0(\xi),b_\mathcal{N}(\xi)\rangle_{\mathbb{C}^4}=0$ thanks to the divergence-free condition $\begin{pmatrix}
\xi \\
0
\end{pmatrix} \cdot \mathscr{F} w = 0$,
we can rewrite the solution to \eqref{FLEQ} as the Duhamel form
\begin{equation*}
   \mathscr{F} w(t,\xi)=\sum_{\gamma=\pm,\mu}e^{t\lambda_\gamma(\xi)}\langle\mathscr{F} w_0(\xi),b_{\gamma}(\xi)\rangle_{\mathbb{C}^4}b_{\gamma}(\xi).\\
\end{equation*}
The above derivation is rigorously restated in the following proposition.
\begin{proposition}
For every $N\geq 0$ and for every $w_0\in L^2(\mathbb{R}^3)$ with $\widetilde{\nabla}\cdot w_0 =0,$ there exists a unique solution $w$ to \eqref{LEQ} which is given by
\begin{equation*}
   w(t,x)=P_\mu w_0(x) + \sum_{\gamma=\pm}e^{\gamma it N p_{\mu}(D) }P_{\gamma} w_0(x)
\end{equation*}
where we set
\begin{equation}\label{def_proj}
   \begin{split}
      P_{\gamma} w_0&:=\mathscr{F}^{-1}[\langle\mathscr{F} w_0(\xi),b_{\gamma}(\xi)\rangle_{\mathbb{C}^4}b_{\gamma}(\xi)], \qquad \gamma=\pm,\mu, \\
       e^{\gamma it N p_{\mu}(D)}\varphi(x)&:=\frac{1}{(2\pi)^3}\int_{\mathbb{R}^3}e^{ix\cdot\xi+ \gamma it N p_{\mu}(\xi)}\mathscr{F}\varphi(\xi)\,\ud\xi, \qquad \gamma=\pm,
   \end{split}
\end{equation}
and
\begin{equation}\label{p_mu}
   p_{\mu}(\xi):=\frac{|\xi_{\mu}|}{|\xi|}.
\end{equation}
\end{proposition}

\section{Proof of Theorem~\ref{thm3}}
For any positive constants $\Omega$ and $N$, by the standard local well-posedness, there exists a time $T \in [0,T_{\ast})$ such that \eqref{eq_main} possess a unique local-in-time solution $u=(v,\theta)$ which satisfies $$u \in C([0,T];H^m(\bbR^3))\cap C^1([0,T];H^{m-1}(\bbR^3)).$$
\begin{proof}[Proof of Theorem~\ref{thm3}]
We assume the hypotheses of Theorem~\ref{thm3}. Fix $s \in \bbN$ with $s \in [3,m-3]$. We have by the continuous embedding $ H^{s}(\bbR^3) \hookrightarrow W^{s-2,\infty}(\bbR^3)$ that
\begin{equation}\label{eng_Ws}
\begin{split}
    \| u(t) - u^{\mu}(t) \|_{W^{s-2,\infty}} &\geq \big\| e^{i N t} P_{+} u_0 + e^{- i N t} P_{-} u_0 \big\|_{W^{s-2,\infty}} - C\sum_{\gamma = \pm} \big\| u^\gamma - e^{\gamma i N t} P_{\gamma} u_0 \big\|_{H^s} - \| \mathcal{E}(t) \|_{H^{s}} \\
    &=: I-II-III
    \end{split}
\end{equation}
for $t\in[0,T],$ where $u^+$ and $u^-$ are defined in \eqref{duhamel_dispersive} as the solutions to the modified linear system \eqref{MLS} and the error $\mathcal{E}$ is defined as $\mathcal{E}=u-u^{+}-u^{-}-u^\mu$. Roughly speaking, our job will be done once we can bound $II$ and $III$ to be as small as we want while $I$ remains big enough. The first step is to control $II$ and the second step is to estimate $III.$ The last step is to prove certain largeness of $I$, which would iron out the whole proof.

\noindent \textbf{Step 1. Decay-in-$N$ of the inhomogeneous parts of $u^{\pm}$}

\noindent Relying on the Duhamel formula
\begin{equation*}
	u^{\pm}(t)=e^{\pm iNt}P_{\pm}u_0 - \int_{0}^{t}e^{\pm i N(t-\tau)}P_{\pm}(u^{\mu}(\tau)\cdot\widetilde{\nabla})u^{\mu}(\tau)\,\mathrm{d}\tau,
\end{equation*}
we claim that the inhomogeneous part can be estimated as
\begin{equation}\label{Ws_upm_est}
\begin{split}
	\sup_{t \in [0,T]} \big\| u^\pm - e^{\pm i N t} P_{\pm} u_0 \big\|_{H^{s+1}} &= \left\| \int_{0}^{t}e^{\pm i N(t-\tau)}P_{\pm}(u^{\mu}(\tau)\cdot\widetilde{\nabla})u^{\mu}(\tau)\,\mathrm{d}\tau \right\|_{H^{s+1}} \\
	&\leq \frac CN (C_L^2 + T C_L^3),
	\end{split}
\end{equation}
where $C_L$ is the constant defined by \eqref{upp_u}. To bring the decay in $N$ out of the time integral, we perform the integration by parts over time as
\begin{gather*}
    \int_{0}^{t}e^{\pm i N(t-\tau)}P_{\pm}(u^{\mu}(\tau)\cdot\widetilde{\nabla})u^{\mu}(\tau)\,\mathrm{d}\tau \\
	= \frac{1}{\pm iN} \int_{0}^{t}e^{\pm i N(t-\tau)}P_{\pm}\big[(\partial_t u^{\mu}(\tau)\cdot\widetilde{\nabla})u^{\mu}(\tau)+ u^{\mu}(\tau)\cdot\widetilde{\nabla})\partial_t u^{\mu}(\tau)\big]\,\mathrm{d}\tau \\
    -\frac{1}{\pm iN} \bigg[e^{\pm i N(t-\tau)}P_{\pm}(u^{\mu}(\tau)\cdot\widetilde{\nabla})u^{\mu}(\tau)\bigg]_{\tau=0}^{t}.
\end{gather*}
Since $H^{s+1}$ is a Banach algebra, the boundary term can be estimated as
\begin{align*}
	\left\| -\frac{1}{\pm iN} \bigg[e^{\pm i N(t-\tau)}P_{\pm}(u^{\mu}(\tau)\cdot\widetilde{\nabla})u^{\mu}(\tau)\bigg]_{\tau=0}^{t} \right\|_{H^{s+1}} &\leq \frac{2}{N} \sup_{t \in [0,T]} \big\|(u^{\mu}(t)\cdot\widetilde{\nabla})u^{\mu}(t)\big\|_{H^{s+1}} \\
	&\leq \frac CN \sup_{t \in [0,T]} \| u^{\mu} \|_{H^{s+2}}^2.
\end{align*}
To control the main time integral, we observe that the integrand consists of the time derivatives of the limit solution $u^\mu$; we need to exploit the limit equations
\begin{equation*}
	\partial_t u^{\mu}+P_{\mu} (u^{\mu} \cdot \widetilde{\nabla}) u^{\mu} = 0
\end{equation*}
to replace the time derivatives by the spatial derivatives. Then we can simply obtain that
\begin{gather*}
	\left\| \frac{1}{\pm iN} \int_{0}^{t}e^{\pm i N(t-\tau)}P_{\pm}\big[(\partial_t u^{\mu}(\tau)\cdot\widetilde{\nabla})u^{\mu}(\tau)+ u^{\mu}(\tau)\cdot\widetilde{\nabla})\partial_t u^{\mu}(\tau)\big]\,\mathrm{d}\tau \right\|_{H^{s+1}} \\
	\leq \frac CN T \sup_{t \in [0,T]} \| u^{\mu} \|_{H^{s+3}}^3
\end{gather*}
where we use the $H^{s+1}$ algebra and the boundedness of $P_\mu$ on $H^{s+1}$. As a consequence, by collecting the above estimates, we arrive at
\begin{equation*}
\begin{split}
	\left\| \int_{0}^{t}e^{\pm i N(t-\tau)}P_{\pm}(u^{\mu}(\tau)\cdot\widetilde{\nabla})u^{\mu}(\tau)\,\mathrm{d}\tau \right\|_{H^{s+1}} &\leq \frac CN \sup_{t \in [0,T]} \| u^{\mu} \|_{H^{s+2}}^2 + \frac CN T \sup_{t \in [0,T]} \| u^{\mu} \|_{H^{s+3}}^3\\
	&\leq \frac{C}{N}(C_L^2+TC_L^3)
	\end{split}
\end{equation*}
thanks to the uniform estimate \eqref{upp_u}. Thus  \eqref{Ws_upm_est} is established as well.

\noindent \textbf{Step 2. Energy estimates for the difference $\mathcal{E}=u-u^{+}-u^{-}-u^\mu$}

\noindent From \eqref{UEQ}, we may compute as
\begin{gather*}
	\frac 12 \frac{\ud}{\ud t} \| \mathcal{E} \|_{H^{s}}^2 + \langle (u \cdot \widetilde{\nabla}) \mathcal{E}, \mathcal{E} \rangle_{H^{s}} + \sum_{j=\mu,\pm} \langle (\mathcal{E} \cdot \widetilde{\nabla}) u^j, \mathcal{E} \rangle_{H^{s}}  + \sum_{\substack{j,k = \mu,\pm \\ (j,k) \neq (\mu,\mu)}} \langle (u^j \cdot \widetilde{\nabla}) u^k, \mathcal{E} \rangle_{H^{s}} = 0.
\end{gather*}
The cancellation property
\begin{gather*}
	\langle (u \cdot \widetilde{\nabla}) \partial^{\alpha}\mathcal{E}, \partial^{\alpha}\mathcal{E} \rangle_{L^2} = 0
\end{gather*}
holds for any multi-index $\alpha$, so we can see that
\begin{align*}
	\big| \langle (u \cdot \widetilde{\nabla}) \mathcal{E}, \mathcal{E} \rangle_{H^{s}} \big| &= \big| \langle (u \cdot \widetilde{\nabla}) \mathcal{E}, \mathcal{E} \rangle_{H^{s}} - \sum_{|\alpha| \leq s} \langle (u \cdot \widetilde{\nabla}) \partial^{\alpha}\mathcal{E}, \partial^{\alpha}\mathcal{E} \rangle_{L^2} \big| \\
	&\leq \sum_{|\alpha|=1}\big| \langle (\partial^{\alpha} u \cdot \widetilde{\nabla}) \mathcal{E}, \partial^{\alpha} \mathcal{E} \rangle_{H^{s-1}} \big| \\
	&\leq C \sum_{j=\mu,\pm} \| u^j \|_{H^{s}} \| \mathcal{E} \|_{H^{s}}^2 + C \| \mathcal{E} \|_{H^{s}}^3.
\end{align*}
In the above computation, the last inequality was derived from that $H^{s-1}(\mathbb{R}^3)$ is a Banach algebra. Similarly, we have
\begin{gather*}
	\sum_{j=\mu,\pm} \big| \langle (\mathcal{E} \cdot \widetilde{\nabla}) u^j, \mathcal{E} \rangle_{H^{s}} \big| \leq C \sum_{j=\mu,\pm} \| u^j \|_{H^{s+1}} \| \mathcal{E} \|_{H^{s}}^2
\end{gather*}
and
\begin{gather*}
	\sum_{\substack{j,k = \mu,\pm \\ (j,k) \neq (\mu,\mu)}} \big| \langle (u^j \cdot \widetilde{\nabla}) u^k, \mathcal{E} \rangle_{H^{s}} \big| \leq C \| \mathcal{E} \|_{H^{s}} \left( \sum_{j=\mu,\pm} \| u^j \|_{H^{s+1}} \right) \left( \sum_{j=\pm} \| u^{j} \|_{H^{s+1}} \right).
\end{gather*}
Combining the above estimates gives
\begin{equation}\label{Hs_E_est}
	\frac{\ud}{\ud t} \| \mathcal{E} \|_{H^{s}} \leq C \left( \| \mathcal{E} \|_{H^{s}} + \sum_{\gamma=\mu,\pm} \| u^\gamma \|_{H^{s+1}} \right) \| \mathcal{E} \|_{H^{s}} + C \left( \sum_{\gamma=\mu,\pm} \| u^\gamma \|_{H^{s+1}} \right) \left( \sum_{\gamma=\pm} \| u^{\gamma} \|_{H^{s+1}} \right),
\end{equation}
which holds for all $t \in (0,T]$.

\noindent \textbf{Step 3. Local-in-time smallness of the difference $\mathcal{E}$}

\noindent In this step, we show that there exists $t_0 \in [0,T]$ such that
\begin{equation}\label{eps_est}
    \sup_{t \in [0,t_0]} \| \mathcal{E} \|_{H^{s}}  \leq \varepsilon \left( \sum_{\gamma = \pm}\| P_{\gamma} u_0 \|_{H^{s+1}} + \frac CN (C_L^2 + TC_L^3) \right)
\end{equation}
for any given $\varepsilon > 0$. Since $\mathcal{E}(0)=0$, we can find a sufficiently small time $t_\ast \in (0,T]$ that satisfies
\begin{equation}\label{E_ass}
	\sup_{t \in [0,t_\ast]} \| \mathcal{E} \|_{H^{s}} \leq \sup_{t \in [0,t_\ast]} \sum_{\gamma=\mu,\pm} \| u^\gamma \|_{H^{s+1}}.
\end{equation}
Then we apply \eqref{Ws_upm_est} to obtain
\begin{equation}\label{triangle}
\begin{split}
    \| u^{\pm} \|_{H^{s+1}} \leq \big\| P_{\pm} u_0 \big\|_{H^{s+1}} + \big\| u^\pm - e^{\pm i N t} P_{\pm} u_0 \big\|_{H^{s+1}} \leq \big\| P_{\pm} u_0 \big\|_{H^{s+1}} + \frac{C}{N}(C_L^2+TC_L^3).
    \end{split}
\end{equation}
Combining the above with \eqref{Hs_E_est}, we see that $\frac{\ud}{\ud t} \| \mathcal{E} \|_{H^{s}}$ is bounded above by a constant $C$ times
\begin{gather*}
	 \left( \sup_{t \in [0,t_\ast]} \sum_{\gamma=\mu,\pm} \| u^\gamma \|_{H^{s+1}} \right) \| \mathcal{E} \|_{H^{s}}
	+ \sup_{t \in [0,t_\ast]} \sum_{\gamma=\mu,\pm} \| u^\gamma \|_{H^{s+1}} \left( \sum_{\gamma = \pm} \| P_{\gamma} u_0 \|_{H^{s+1}} + \frac CN (C_L^2 + TC_L^3) \right).
\end{gather*}
The Gr\"{o}nwall inequality gives us
\begin{equation}\label{E}
\begin{split}
	\sup_{t \in [0,t_\ast]} \| \mathcal{E} \|_{H^{s}}
	\leq Ct_\ast &\sup_{t \in [0,t_\ast]} \sum_{\gamma=\mu,\pm} \| u^\gamma \|_{H^{s+1}} \\
 \times \left( \sum_{\gamma=\pm} \| P_{\gamma} u_0 \|_{H^{s+1}} + \frac CN (C_L^2 +  TC_L^3) \right) &\exp \left(Ct_\ast \sup_{t \in [0,t_\ast]} \sum_{\gamma=\mu,\pm} \| u^\gamma \|_{H^{s+1}} \right).
 \end{split}
\end{equation}
Due to \eqref{upp_u}, \eqref{Ws_upm_est}, and \eqref{triangle}, it is immediately obtained that
\begin{gather*}
	\sup_{t \in [0,t_\ast]} \sum_{\gamma=\mu,\pm} \| u^\gamma \|_{H^{s+1}} \leq C_L + \frac CN(C_L^2 + TC_L^3) + \sum_{\gamma=\pm}\| P_{\gamma} u_0 \|_{H^{s+1}},
\end{gather*}
and thus, for any given $\varepsilon \in (0,1)$, we may choose a sufficiently small $t_0=t_0(\varepsilon) \in (0,t_\ast]$ such that
\begin{equation*}
	Ct_0 \sup_{t \in [0,t_0]} \sum_{\gamma=\mu,\pm} \| u^  \gamma \|_{H^{s+1}} \exp \left( Ct_0 \sup_{t \in [0,t_0]} \sum_{\gamma=\mu,\pm} \| u^ \gamma \|_{H^{s+1}} \right)\leq \varepsilon.
\end{equation*}
Then \eqref{eps_est} follows from \eqref{E}.

\noindent \textbf{Step 4. Preservation of largeness from orthogonality}

We first show a uniform-in-$N$ lower bound of the linear propagation from initial data. More precisely, we prove that there exists a constant $A>0$ such that
\begin{equation}\label{A_est}
	\Big\| \sum_{\gamma=\pm}e^{i\gamma N t} P_{\gamma} u_0 \Big\|_{W^{s-2,\infty}} \geq A
\end{equation}
for all $N$ and $t \in [0,T]$. Let $r>0$. Then, we can see
\begin{equation*}
    \Big\| \sum_{\gamma=\pm}e^{i\gamma N t} P_{\gamma} u_0 \Big\|_{W^{s-2,\infty}} \geq \Big\| \sum_{\gamma=\pm}e^{i\gamma N t} P_{\gamma} u_0 \Big\|_{L^{\infty}(B(0;r))} \geq |B(0;r)|^{-\frac{1}{2}} \Big\| \sum_{\gamma=\pm}e^{i\gamma N t} P_{\gamma} u_0 \Big\|_{L^2(B(0;r))},
\end{equation*}
where $|B(0;r)|$ denotes the volume of the ball centered at the origin with its radius $r$. Since the quantity
\begin{equation*}
    \left| \sum_{\gamma=\pm}e^{i\gamma N t} P_{\gamma} u_0 \right|^2 = |P_{+} u_0|^2 + |P_{-} u_0|^2 + 2\operatorname{Re}[e^{2iNt}\langle P_{+} u_0, P_{-} u_0 \rangle]
\end{equation*}
enjoys certain decay at spatial infinity due to the initial integrability condition $u_0\in H^m(\bbR^3)$, we can find a constant $\varepsilon = \varepsilon(r)>0$ not depending on $N$ and $t$ such that $\varepsilon(r) \to 0$ as $r \to \infty$ and 
\begin{equation*}
    \left\| \sum_{\gamma=\pm}e^{i\gamma N t} P_{\gamma} u_0 \right\|_{L^2(B(0;r))} \geq \left\| \sum_{\gamma=\pm}e^{i\gamma N t} P_{\gamma} u_0 \right\|_{L^2(\bbR^3)} - \varepsilon.
\end{equation*}
Observing that $\mathscr{F}P_{+} u_0(\xi)$ and $\mathscr{F}P_{-} u_0(\xi)$ are orthogonal in $\bbC^4$ for every $\xi\in\bbR^3$, we have by the Plancherel theorem $\big\| \sum_{\gamma=\pm}e^{i\gamma N t} P_{\gamma} u_0 \big\|_{L^2} > 0$ for all $t\in \mathbb{R}$, except for the case $P_{+} u_0 = P_{-} u_0 = 0$. Taking $r>0$ sufficiently large, we can have $\big\| \sum_{\gamma=\pm}e^{i\gamma N t} P_{\gamma} u_0 \big\|_{L^2} - \varepsilon \geq \frac{1}{2} \big\| \sum_{\gamma=\pm}e^{i\gamma N t} P_{\gamma} u_0 \big\|_{L^2}$. Thus, we obtain \eqref{A_est} with 
\begin{equation*}
    A := \frac{1}{2} |B(0;r)|^{-\frac{1}{2}} \Big\| \sum_{\gamma=\pm}e^{i\gamma N t} P_{\gamma} u_0 \Big\|_{L^2}.
\end{equation*}
Now we are ready to finish the proof. From the above ingredients \eqref{eng_Ws}, \eqref{Ws_upm_est}, \eqref{eps_est}, and \eqref{A_est}, we may deduce that
\begin{align*}
	&\inf_{t \in [0,t_0]} \| u(t) - u^{\mu}(t) \|_{W^{s-2,\infty}} \\
	&\hphantom{\qquad\qquad}\geq \inf_{t \in [0,t_0]} \left\| \sum_{\gamma=\pm} e^{\gamma i Nt} P_{\gamma} u_0 \right\|_{W^{s-2,\infty}} - \sum_{\gamma=\pm} \sup_{t \in [0,t_0]} \| u^{\gamma}(t) - e^{\gamma i Nt} P_{\gamma} u_0 \|_{H^s} - \sup_{t \in [0,t_0]} \| \mathcal{E} \|_{H^{s}} \\
	&\hphantom{\qquad\qquad}\geq A - \frac CN (C_L^2 + TC_L^3) - \varepsilon \left(  \sum_{\gamma=\pm} \| P_{\gamma} u_0 \|_{H^{s+1}} + \frac CN (C_L^2 + TC_L^3) \right) \\
	&\hphantom{\qquad\qquad}= A - \varepsilon \sum_{\gamma=\pm} \| P_{\gamma} u_0 \|_{H^{s+1}} - \frac CN (C_L^2 + TC_L^3).
\end{align*}
Here, we fix $\varepsilon$ to satisfy $\varepsilon \sum_{\gamma=\pm} \| P_{\gamma} u_0 \|_{H^{s+1}} \leq \frac 14 A$. Then, taking $R>0$ with $ \frac C{R} (C_L^2 + TC_L^3)  \leq  \frac 14 A$, we obtain
\begin{align*}
	\inf_{t \in [0,t_0]} \| u(t) - u^{\mu}(t) \|_{W^{s-2,\infty}} \geq \frac 12 A
\end{align*}
for all $N > R$. This completes the proof. 
\end{proof}

\section{Proofs of Theorem~\ref{dich} and Corollary~\ref{cor_div}}
\begin{proof}[Proof of Theorem~\ref{dich}]By the triangle inequality, the convergence rate \eqref{conv_1}, the equivalence \eqref{u_phi}, and Proposition~\ref{conv_prop}, we get
\begin{equation}
\begin{split}
    \|u^{(k)}-\nabla_1 \psi^{1}\|_{L^ q(0,T;W^{1,\infty})} &\leq \|u^{(k)}-\nabla_{\mu_k}\psi^{\mu_k}\|_{L^ q(0,T;W^{1,\infty})} + \|\nabla_{\mu_k}\psi^{\mu_k} - \nabla_1 \psi^{1}\|_{L^ q(0,T;W^{1,\infty})} \\
    &\leq \frac{C_{\mu_k}}{\sqrt{\Omega_k^2+N_k^2}} + \|\nabla_{\mu_k}\psi^{\mu_k} - \nabla_1 \psi^{1}\|_{L^ q(0,T;H^{3})} \\
    &\leq \frac{1}{h(k)\sqrt{1+\mu_k^2}} + CT^{\frac{1}{q}}|\mu_k-1|
\end{split}
\end{equation}
for some function $h\in C^1(\bbR_{+})$ with $h(x)\to\infty$ as $x\to \infty$. Here we also used the continuous embedding $H^3 \hookrightarrow W^{1,\infty}$. Since $\mu_k\to 1$ as $k\to\infty$, the first part \eqref{valid} of Theorem~\ref{dich} follows as desired.

We consider an intermediate sequence $\{(\widetilde{\Omega}_k,\widetilde{N}_k)\}_{k\in\bbN}$ with the fixed ratio $\frac{\widetilde{\Omega}_k}{\widetilde{N}_k}=1$ for any $k\in \bbN$; in particular, we set $\widetilde{\Omega}_k:=\Omega_k$.  Let $\widetilde{u}^{(k)}$ be the corresponding unique classical solution to \eqref{eq_main} on $[0,T]$ with $\Omega=\widetilde{\Omega}_k$ and $N=\widetilde{N}_k$ for each $k\in\bbN$. Then the triangular inequality gives us
\begin{equation}
    \begin{split}
        \|u^{(k)}-\nabla_1 \psi^1\|_{L^q(0,T;W^{1,\infty})} & \geq \|u^{(k)}-\nabla_1 \psi^1\|_{L^q(0,t_0;W^{1,\infty})} \\ &\geq \|\widetilde{u}^{(k)}-\nabla_1 \psi^1\|_{L^q(0,t_0;W^{1,\infty})} - \| u^{(k)}-\widetilde{u}^{(k)}\|_{L^q(0,t_0;W^{1,\infty})} \\
        &\geq \frac{1}{2}At_0^{\frac{1}{q}} -  C(N_k-\widetilde{N}_k)t_0^{\frac{1}{q}}
    \end{split}
\end{equation}
for some $t_0\in(0,T_{\ast})$, where $A=A(u_0)$ is independent of $k$ and $T_{\ast}=T_{\ast}(m,\|u_0\|_{H^m})$ is the maximal existence time, independent of $k$ as well. In the above estimation, we use Theorem~\ref{thm3} and Proposition~\ref{loc_diff}. Since $\widetilde{N}_k=\widetilde{\Omega}_k=\Omega_k=N_k\mu_k$ implies $N_k-\widetilde{N}_k = (1-\mu_k)N_k$, the slow condition $N_k=o(|1-\mu_k|^{-1})$ yields the convergence $|N_k-\widetilde{N}_k|\to 0$ as $k\to\infty.$ Then the second part \eqref{invalid} follows as well. This finishes the proof of Theorem~\ref{dich}.
\end{proof}

\begin{proof}[Proof of Corollary~\ref{cor_div}] Take any sequence $\{ \mu_k \}_{k \in \bbN}$ such that $$\mu_k \to 1 \qquad \mbox{and} \qquad C_{\mu_k}|1-\mu_k| \to 0 \qquad \mbox{as} \qquad k \to \infty.$$ We define $$N_k := \frac{\sqrt{C_{\mu_k}}}{\sqrt{|1-\mu_k|}} \qquad \mbox{and} \qquad \Omega_{k} := N_{k} \mu_{k}.$$ Assume for a contradiction that $h(k) = 1/\sqrt{C_{\mu_k} |1-\mu_k|}$. Then, we can see $$\frac{N_{k}}{C_{\mu_k}} = h(k) \to \infty \qquad \mbox{at the same time} \qquad N_k|1-\mu_k| = \frac{1}{h(k)}\to 0 \qquad \mbox{as} \qquad k \to \infty.$$
Then the sequence $\{(\Omega_l,N_l)\}_{l\in\bbN}$ satisfies both \eqref{fast} and \eqref{slow}, yielding the results \eqref{valid} and \eqref{invalid}, simultaneously. This is a contradiction, which finishes the proof.
\end{proof}

\section{Convergence and non-convergence in $H^{m-3}$ for $\mu\in(0,\infty)$}

\subsection{Proof of Theorem~\ref{rmk_div}}
\begin{proof}[Proof of Theorem~\ref{rmk_div}]
Assume that the hypotheses for Theorem~\ref{rmk_div} are true. Fix any $s \in \bbN$ with $s \in [3,m-3]$. A simple triangular inequality holds as
\begin{equation}\label{eng_Hs}
    \| u(t) - u^{\mu}(t) \|_{H^s} \geq \left\|\sum_{\gamma=\pm}e^{\gamma it N P_{\mu}(D)} P_{\gamma} u_0\right\|_{H^s} - \left\| \sum_{\gamma=\pm} (u^\gamma - e^{\gamma i t N P_{\mu}(D)} P_{\gamma} u_0)\right\|_{H^s} - \| \mathcal{E}(t) \|_{H^{s}},
\end{equation}
for all $t\in[0,T].$

\noindent \textbf{Step 1. Decay-in-$N$ of the inhomogeneous parts of $u^{\pm}$}

\noindent Our claim is that
\begin{equation}\label{Hs_upm_est}
	\sup_{t \in [0,T]} \left\| \sum_{\gamma=\pm} (u^\gamma - e^{\gamma i t N P_{\mu}(D)} P_{\gamma} u_0) \right\|_{H^{s+1}} \leq \frac CN (C_L^2 + TC_L^3).
\end{equation}
To prove the claim, we rely on the Duhamel formula
\begin{equation*}
	u^{\pm}(t)=e^{\pm itNp_\mu(D)}P_{\pm}u_0 - \int_{0}^{t}e^{\pm i (t-\tau)N p_\mu(D)}P_{\pm}(u^{\mu}(s)\cdot\widetilde{\nabla})u^{\mu}(s)\,\mathrm{d}\tau
\end{equation*}
which reduces \eqref{Hs_upm_est} to the following decay-in-$N$ estimate
\begin{equation*}
	\left\| \int_{0}^{t}e^{\pm i (t-\tau)N p_\mu(D)}P_{\pm}(u^{\mu}(\tau)\cdot\widetilde{\nabla})u^{\mu}(\tau)\,\mathrm{d}\tau \right\|_{H^{s+1}} \leq \frac CN (C_L^2 + TC_L^3).
\end{equation*}
Taking the Fourier transform in $\bbR^3$ and performing the integration by parts over time, we compute
\begin{gather*}
    \int_{0}^{t}e^{\pm i (t-\tau)N\frac {|\xi_{\mu}|}{|\xi|}} \langle (u^{\mu}(\tau)\cdot\widetilde{\nabla})u^{\mu}(\tau), b_{\pm} \rangle b_{\pm} \,\mathrm{d}\tau \\
	= \frac{|\xi|}{\pm iN|\xi_{\mu}|} \int_{0}^{t}e^{\pm i (t-\tau)N\frac {|\xi_{\mu}|}{|\xi|}}P_{\pm}\big[(\partial_t u^{\mu}(\tau)\cdot\widetilde{\nabla})u^{\mu}(\tau)+ (u^{\mu}(\tau)\cdot\widetilde{\nabla})\partial_t u^{\mu}(\tau)\big]\,\mathrm{d}\tau \\
    -\frac{|\xi|}{\pm iN|\xi_{\mu}|} \bigg[e^{\pm i (t-\tau)N\frac {|\xi_{\mu}|}{|\xi|}}P_{\pm}(u^{\mu}(\tau)\cdot\widetilde{\nabla})u^{\mu}(\tau)\bigg]_{\tau=0}^{t}.
\end{gather*}
We use the fact $\frac {|\xi|}{|\xi_{\mu}|} \leq \max \{1,\mu^{-1} \}$ to estimate the boundary term as
\begin{align*}
	\left\| -\frac{|\xi|}{\pm iN|\xi_{\mu}|} \bigg[e^{\pm i (t-\tau)N\frac {|\xi_{\mu}|}{|\xi|}}P_{\pm}(u^{\mu}(\tau)\cdot\widetilde{\nabla})u^{\mu}(\tau)\bigg]_{\tau=0}^{t} \right\|_{H^{s+1}} &\leq \frac{C}{N} \sup_{t \in [0,T]} \big\|(u^{\mu}(t)\cdot\widetilde{\nabla})u^{\mu}(t)\big\|_{H^{s+1}} \\
	&\leq \frac CN \sup_{t \in [0,T]} \| u^{\mu} \|_{H^{s+2}}^2.
\end{align*}
It remains to control the main time integral; to deal with the time derivatives of $u^\mu$, we invoke the equations
\begin{equation*}
	\partial_t u^{\mu}+P_{\mu} (u^{\mu} \cdot \widetilde{\nabla}) u^{\mu} = 0.
\end{equation*}
The limit equations replace the time derivatives by the nonlinearity involving space derivatives only, and so we get
\begin{gather*}
	\left\| \frac{|\xi|}{\pm iN|\xi_{\mu}|} \int_{0}^{t}e^{\pm i (t-\tau)N\frac {|\xi_{\mu}|}{|\xi|}}P_{\pm}\big[(\partial_t u^{\mu}(\tau)\cdot\widetilde{\nabla})u^{\mu}(\tau)+ (u^{\mu}(\tau)\cdot\widetilde{\nabla})\partial_t u^{\mu}(\tau)\big]\,\mathrm{d}\tau \right\|_{H^{s+1}} \\
	\leq \frac CN T \sup_{t \in [0,T]} \| u^{\mu} \|_{H^{s+3}}^3,
\end{gather*}
where we used $H^{s+1}$ algebra and boundedness of $P_{\pm}$ in $H^{s+1}.$ From the previous estimates, we may deduce \eqref{Hs_upm_est} by the use of the Plancherel theorem.

\noindent \textbf{Step 2. Local-in-time smallness of the difference $\mathcal{E}$}

\noindent Following the way we showed \eqref{eps_est}, we use \eqref{Hs_upm_est} and find $t_0 \in (0,T]$ satisfying \begin{equation*}
    \sup_{t \in [0,t_0]} \| \mathcal{E} \|_{H^{s}}  \leq \varepsilon \left( \| P_{\pm} u_0 \|_{H^{s+1}} + \frac CN (C_L^2 + TC_L^3) \right),
\end{equation*}
where $\varepsilon$ is an arbitrary positive number.

\noindent \textbf{Step 3. Preservation of largeness via orthogonality}

\noindent It is clear that by orthogonality we have for any $\ell \in [0,m]$ that
\begin{equation}\label{Hs_initial}
	\left\| \sum_{\gamma = \pm} e^{\gamma i t N P_{\mu}(D)} P_{\gamma} u_0 \right\|_{H^\ell}^2 = \sum_{\gamma = \pm} \left\| P_{\gamma} u_0 \right\|_{H^\ell}^2 = \big\| u_0 - P_\mu u_0 \big\|_{H^\ell}^2, \qquad t \in [0,T].
\end{equation}
\noindent We are ready to prove Theorem \ref{rmk_div}. Collecting the above estimates, we deduce from \eqref{eng_Hs} that
\begin{align*}
	\inf_{t \in [0,t_0]} \| u - u^{\mu} \|_{H^{s}} &\geq \| u_0 - P_\mu u_0 \|_{H^s} - \sum_{\gamma=\pm} \sup_{t \in [0,t_0]} \big\| u^\gamma - e^{\gamma i t N P_{\mu}(D)} P_{\gamma} u_0 \big\|_{H^s} - \sup_{t \in [0,t_0]} \| \mathcal{E}(t) \|_{H^{s}} \\
	&\geq \| u_0 - P_\mu u_0 \|_{H^s} - \frac CN (C_L^2 + TC_L^3) - \varepsilon \left(  \sum_{\gamma=\pm} \| P_{\gamma} u_0 \|_{H^{s+1}} + \frac CN (C_L^2 + TC_L^3)\right) \\
	&= \| u_0 - P_\mu u_0 \|_{H^s} -\varepsilon \sum_{\gamma=\pm} \| P_{\gamma} u_0 \|_{H^{s+1}} - \frac CN (C_L^2 + TC_L^3).
\end{align*}
We fix $\varepsilon$ to satisfy $\varepsilon \sum_{\gamma=\pm} \| P_{\gamma} u_0 \|_{H^{s+1}} \leq \frac 14 \| u_0 - P_\mu u_0 \|_{H^{s}}$. Then, taking $R>0$ with $ \frac C{R}(C_L^2 + TC_L^3) \leq  \frac 14 \| u_0 - P_\mu u_0 \|_{H^s}$, we obtain
\begin{align*}
	\inf_{t \in [0,t_0]} \| u - u^{\mu} \|_{H^{s}} \geq \frac 12 \| u_0 - P_\mu u_0 \|_{H^{s}}
\end{align*}
for all $N > R$. This completes the proof.
\end{proof}
\subsection{Proof of Theorem~\ref{thm2}}
\begin{proof}[Proof of Theorem~\ref{thm2}]
Assume that the hypotheses for Theorem~\ref{thm2} are true. For any given $\Omega>0$ and $N>0$, by the standard local well-posedness, there exists $T_\ast>0$ such that \eqref{eq_main} possesses a unique local-in-time solution $u$ satsifying $$u \in C([0,T_\ast];H^m(\bbR^3))\cap C^1((0,T_\ast);H^{m-1}(\bbR^3)).$$ To show that the local-in-time solution can be uniquely extended to the entire $[0,T]$, thanks to the blow-up criterion of \eqref{eq_main}, it suffices to obtain that
\begin{equation}\label{1.8claim}
	\sup_{t \in [0,T]} \| u \|_{H^{s}} < \infty, \qquad s > \frac 52.
\end{equation}
We compute
\begin{equation}
\begin{split}
\| u \|_{H^{s}} &\leq \| u^{\mu} \|_{H^{s}} + \left\| \sum_{\gamma=\pm} e^{\gamma itN p_{\mu}(D)} P_{\gamma} u_0 \right\|_{H^{s}} + \sum_{\gamma=\pm}\| u^{\gamma} - e^{\gamma itN p_{\mu}(D)} P_{\gamma} u_0 \|_{H^{s}} + \| \mathcal{E} \|_{H^{s}}\\
&\leq C_L + \left\| u_0 - P_\mu u_0 \right\|_{H^s} + \frac{C}{N}(C_L^2+TC_L^3) +\|\mathcal{E}\|_{H^s} \\
&\leq C + \|\mathcal{E}\|_{H^s}
\end{split}
\end{equation}
for some $C=C(\mu,T,N,\|u_0\|_{H^{m}}),$ due to \eqref{upp_u}, \eqref{Hs_upm_est}, and \eqref{Hs_initial} that were previously obtained. Thus showing \eqref{1.8claim} is equivalent to prove $$
\sup_{t\in[0,T]}\|\mathcal{E}\|_{H^s}<\infty.$$ To this end, we need to perform the standard $H^s$ estimate of $\mathcal{E}$ for $s \in \bbN$ satisfying $s \in [3,m-3]$. Since $\mathcal{E}(0)=0,$ we may choose a possibly small $t_0 \in (0,T]$ to satisfy
\begin{equation}\label{E_ass_2}
	\sup_{t \in [0,t_0]} \| \mathcal{E} \|_{H^{s}} \leq \sup_{t \in [0,T]} \sum_{\gamma=\mu,\pm} \| u^\gamma \|_{H^{s+1}}.
\end{equation}
This leads to, with the help of \eqref{Hs_E_est} and \eqref{Hs_upm_est}, the following $H^s$ energy inequality
\begin{gather*}
	\frac{\ud}{\ud t} \| \mathcal{E} \|_{H^{s}} \leq C \left( \sup_{t \in [0,T]} \sum_{\gamma=\mu,\pm} \| u^\gamma \|_{H^{s+1}} \right) \| \mathcal{E} \|_{H^{s}} \\
	+ C \sup_{t \in [0,T]} \sum_{\gamma=\mu,\pm} \| u^\gamma \|_{H^{s+1}} \left( \sum_{\gamma=\pm} \| P_{\gamma} u_0 \|_{H^{s+1}} + \frac CN (C_L^2 + TC_L^3) \right).
\end{gather*}
The Gr\"{o}nwall inequality further implies
\begin{gather*}
	\sup_{t \in [0,t_0]} \| \mathcal{E} \|_{H^{s}} 
	\leq CT \sup_{t \in [0,T]} \sum_{\gamma=\mu,\pm} \| u^\gamma \|_{H^{s+1}}  \\
 \times \left( \sum_{\gamma=\pm} \| P_{\gamma} u_0 \|_{H^{s+1}} + \frac CN (C_L^2 + TC_L^3) \right) \exp \left(CT \sup_{t \in [0,T]} \sum_{\gamma=\mu,\pm} \| u^\gamma \|_{H^{s+1}} \right).
\end{gather*}
It is enough to estimate the right-hand side. Here the assumption $\sum_{\gamma=\pm} CTe^{CT} \| P_{\gamma} u_0 \|_{H^{s+1}} \leq \frac{1}{2}$ kicks in importantly; it allows us to take a sufficiently large $R>0$ such that
\begin{align*}
	CT \left( \sum_{\gamma=\pm} \| P_{\gamma} u_0 \|_{H^{s+1}} + \frac CR (C_L^2 + TC_L^3) \right) \exp \left(CT \sup_{t \in [0,T]} \sum_{\gamma=\mu,\pm} \| u^\gamma \|_{H^{s+1}} \right) \leq \frac{1}{2}.
\end{align*}
This implies that we have
\begin{equation}\label{E_3}
    \sup_{t \in [0,t_0]} \| \mathcal{E} \|_{H^{s}} \leq \frac{1}{2}\sup_{t \in [0,T]} \sum_{\gamma=\mu,\pm} \| u^\gamma \|_{H^{s+1}}.
\end{equation}
holds for all $N > R$. To encapsulate, the initial smallness \eqref{E_ass_2} induces the further smallness \eqref{E_3}, which is the applicable form for a standard continuity argument. It turns out that we can choose $t_0 = T$, i.e.,
\begin{equation}\label{eps_Hs_est}
	\sup_{t \in [0,T]} \| \mathcal{E} \|_{H^{s}} \leq \sup_{t \in [0,T]} \sum_{\gamma=\mu,\pm} \| u^\gamma \|_{H^{s+1}}.
\end{equation}
Combining the above estimates, we conclude that \eqref{1.8claim} holds as desired.

Now we prove the convergence part. Note that we additionally assumed that $P_{\pm} u_0 = 0$ for the convergence regardless of $\mu.$ Such an assumption leads to
\begin{equation}\label{sing_lim_eng}
    \sup_{t \in [0,T]} \| u - u^{\mu} \|_{H^{s}} \leq \sum_{\gamma=\pm} \sup_{t \in [0,T]} \| u^{\gamma} \|_{H^{s}} + \sup_{t \in [0,T]} \| \mathcal{E} \|_{H^{s}}.
\end{equation} by a simple triangle inequality. From \eqref{upp_u}, \eqref{Hs_E_est}, \eqref{Hs_upm_est}, and \eqref{eps_Hs_est}, we obtain the $H^s$ estimate
\begin{equation*}
\begin{split}
    \frac{\ud}{\ud t} \| \mathcal{E} \|_{H^{s}} &\leq C \sum_{\gamma=\mu,\pm} \| u^\gamma \|_{H^{s+1}} \| \mathcal{E} \|_{H^{s}} + C \left( \sum_{\gamma=\pm}\| u^{\gamma} \|_{H^{s+1}} \right) \left( \sum_{\gamma=\mu,\pm} \| u^\gamma \|_{H^{s+1}} \right) \\
    &\leq C\|\mathcal{E}\|_{H^{s}}+\frac{C}{N}
\end{split}
\end{equation*}
for some $C=C(m,\mu,T,\|u_0\|_{H^m})$ as long as $N>R$, where $R$ is a positive constant obtained in the proof of the existence part. The Gr\"{o}nwall inequality gives
\begin{align*}
	\sup_{t \in [0,T]} \| \mathcal{E} \|_{H^{s}} &\leq \frac{CT}{N} \exp (CT).
\end{align*}
Using \eqref{sing_lim_eng}, we establish that
\begin{align*}
	\sup_{t \in [0,T]} \| u - u^{\mu} \|_{H^{s}} &\leq \frac{CT}{N} + \frac{CT}{N} e^{CT} \leq \frac CN
\end{align*}
for some $C=C(m,\mu,T,\|u_0\|_{H^m})$ as desired. This finishes the proof.
\end{proof}

\appendix

\section{Appendix}

\subsection{Continuity in $N$ for the Boussinesq equations}
\begin{proposition}\label{loc_diff}
    Let $(\Omega,N)$ and $(\widetilde{\Omega},\widetilde{N})$ be pairs of positive constants with $\Omega = \widetilde{\Omega}$, and $u_0$ be the initial data belonging to $H^m(\bbR^3)$ for some $m>\frac{5}{2}$. Let $u$ and $\widetilde{u}$ be the corresponding local-in-time solutions with the maximal time $T_{\ast}>0$ such that \eqref{loc_est} holds. Then, for any $T \in (0,T_{\ast})$, there exists $C=C(m,T_\ast)>0$ such that 
    \begin{equation}\label{conv_tilde}
        \sup_{t \in [0,T]} \| u(t) - \widetilde{u}(t) \|_{H^{m-1}} \leq C(N-\widetilde{N}) e^{C\| u_0 \|_{H^m}T}.
    \end{equation}
\end{proposition}
\begin{proof}
   From our system \eqref{eq_main} we derive that $\partial_t(v - \widetilde{v})$ equals
    \begin{gather*}
          -((v - \widetilde{v}) \cdot \nabla)v - (\widetilde{v} \cdot \nabla) (v - \widetilde{v}) - \Omega e_3 \times (v - \tilde{v})
        - \nabla(p - \widetilde{p}) + (N - \widetilde{N}) \tht e_3 + \widetilde{N}(\tht - \widetilde{\tht})
    \end{gather*}
    and $\partial_t(\tht - \widetilde{\tht})$ equals
    \begin{gather*}
          -((v - \widetilde{v}) \cdot \nabla)\tht - (\widetilde{v} \cdot \nabla) (\tht - \widetilde{\tht}) - (N - \widetilde{N}) v_3 - \widetilde{N}(v_3 - \widetilde{v}_3).
    \end{gather*}
    Thus, we can see for any $k \in \bbN \cup \{ 0 \}$ with $0 \leq k \leq m-1$ that $\frac{1}{2} \frac{\ud}{\ud t} \| u - \widetilde{u} \|_{\dot{H}^{k}}^2$ is bounded above by
    \begin{gather*}
         \sum_{|\alpha| = k} \int |\partial^{\alpha}[((v - \widetilde{v}) \cdot \nabla)v]| |\partial^{\alpha}(v - \widetilde{v})| \,\ud x 
        - \sum_{|\alpha| = k} \int \partial^{\alpha}[(\widetilde{v} \cdot \nabla) (v - \widetilde{v})] \cdot \partial^{\alpha} (v - \widetilde{v}) \,\ud x \\
        + \sum_{|\alpha| = k} \int |\partial^{\alpha}[((v - \widetilde{v}) \cdot \nabla)\tht]| |\partial^{\alpha}(\tht - \widetilde{\tht})| \,\ud x 
        - \sum_{|\alpha| = k} \int \partial^{\alpha}[(\widetilde{v} \cdot \nabla) (\tht - \widetilde{\tht})] \cdot \partial^{\alpha} (\tht - \widetilde{\tht}) \,\ud x \\
        + |N - \widetilde{N}| \sum_{|\alpha| = k} \left( \int \partial^{\alpha}\tht \partial^{\alpha}(v - \widetilde{v}) \,\ud x + \int \partial^{\alpha}v_d \partial^{\alpha}(\tht - \widetilde{\tht}) \,\ud x \right).
        \end{gather*}
        Using \eqref{loc_est}, H\"older's inequality and the continuous embedding $H^{m-1}(\bbR^3) \hookrightarrow L^{\infty}(\bbR^3)$, the first and third integral are bounded by $C \| u_0 \|_{H^m} \| u - \widetilde{u} \|_{H^{m-1}}^2$. With the integration by parts, the second and fourth integral have same upper bound. On the other hand, it holds
        \begin{gather*}
        |N - \widetilde{N}| \sum_{|\alpha| = k} \left| \int \partial^{\alpha}\tht \partial^{\alpha}(v - \widetilde{v}) \,\ud x + \int \partial^{\alpha}v_d \partial^{\alpha}(\tht - \widetilde{\tht}) \,\ud x \right| \leq C(N- \widetilde{N}) \| u_0 \|_{H^{m-1}} \| u - \widetilde{u} \|_{H^{m-1}}.
        \end{gather*}
        Thus, we deduce that
        \begin{gather*}
        \frac{1}{2} \frac{\ud}{\ud t} \| u - \widetilde{u} \|_{H^{m-1}}^2 \leq C \| u_0 \|_{H^m} \| u - \widetilde{u} \|_{H^{m-1}} (\| u - \widetilde{u} \|_{H^{m-1}} + (N- \widetilde{N})).
        \end{gather*}
        By Gr\"onwall's inequality, \eqref{conv_tilde} is obtained. This completes the proof.
\end{proof}

\subsection{Global well-posedness of the QG equations}\label{GE_QG}

In this subsection, we formally derive the QG equations as the limit equations, and prove that the solutions to the QG equations are well-defined up to the arbitrary fixed time $T$ for any $H^s$ initial data with any real number $s>\frac{5}{2}$. This is necessary to validate the convergence process for the entire given time interval $[0,T]$. Note that a similar global existence result was obtained in \cite{BB} for $H^s$ with integer $s\geq 3$ on a box-shaped domain.

We first give a formal derivation of the limit equations. Fix any $\mu \in (0,\infty)$. Recall from \eqref{def_b} and \eqref{def_proj} that 
$$
P_\mu w = \int e^{2\pi i x \cdot \xi}\langle \mathscr{F} w(\xi),b_\mu(\xi)\rangle_{\bbC^{4}} b_\mu(\xi) \, \mathrm{d}\xi, \qquad b_\mu(\xi) = \frac {\xi_\mu}{|\xi_\mu|},
$$
where $\xi_{\mu}=(-\xi_2,\xi_1,0,\mu \xi_3)^{\top}.$ Then we can write the limit system as
\begin{equation}\label{LS}
\left\{
\begin{array}{ll}
u^\mu_t + P_\mu[(u^\mu \cdot \widetilde{\nabla})u^\mu] = 0, \qquad P_\mu u^\mu = u^\mu,\\
u^\mu(0,x) = P_\mu u_0(x).
\end{array}
\right.
\end{equation}
To justify the above expression as the limit QG system, we show the formal equivalence between \eqref{LS} and the QG equations \eqref{QG}.
If we define $\psi^\mu$ by $$\psi^\mu := -(-\Delta_\mu)^{-1} \nabla_\mu \cdot u^\mu 
,$$ where $\Delta_\mu = \partial_1^2+\partial_2^2+\mu^2\partial_3^2$ and $\nabla_{\mu} = (-\partial_2,\partial_1,0,\mu\partial_3)^{\top}$, then the identity
\begin{equation}\label{u_phi}
u^\mu = P_\mu{u}^\mu = -\nabla_\mu (-\Delta_\mu)^{-1} \nabla_\mu \cdot u^\mu = \nabla_\mu \psi^{\mu},
\end{equation}
allows us to compute as
\begin{align*}
\nabla_\mu \cdot P_\mu[(u^\mu \cdot \widetilde{\nabla})u^\mu] = \nabla_\mu \cdot P_\mu[(\nabla_\mu \psi^{\mu} \cdot \widetilde{\nabla})\nabla_\mu \psi^{\mu}] = \nabla_\mu \cdot (\nabla_\mu \psi^{\mu} \cdot \widetilde{\nabla})\nabla_\mu \psi^{\mu}. 
\end{align*}
Once we recall \eqref{proj_notation} for the definition of $\widetilde{\nabla}$, a direct computation gives
\begin{align*}
\nabla_\mu \cdot (\nabla_\mu \psi^{\mu} \cdot \widetilde{\nabla})\nabla_\mu \psi^{\mu}= (\nabla_\mu \psi^{\mu} \cdot \widetilde{\nabla})\Delta_\mu \psi^{\mu} = (\nabla_{H}^{\perp} \psi^{\mu} \cdot \nabla_H)\Delta_\mu \psi^{\mu}.
\end{align*}
We see that $\psi^\mu$ satisfies \eqref{QG} as
$$
\left\{
\begin{array}{ll}
\Delta_{\mu} \psi^{\mu}_t +(\mathbf{v}_{H}\cdot \nabla_{H})\Delta_{\mu} \psi^{\mu} =0, \qquad \mathbf{v}_{H} := \nabla_{H}^{\perp} \psi^{\mu}, \\
\nabla_{\mu}\psi^{\mu}(0,x) = u^\mu(0,x) = P_\mu u_0(x).
\end{array}
\right.
$$

\noindent  The above formal argument can be rigorously justified with a minor modification, which we omit here. Then it suffices to show that the system \eqref{LS} is well-posed for arbitrarily long time.

\begin{proposition}\label{uniform}
Let $m > 5/2$ and $\mu \in (0,\infty)$. Then for any $u^{\mu}_0 \in H^m(\mathbb{R}^3)$ and $T > 0$, the limit system \eqref{LS} possesses a unique classical solution
$$
u^\mu \in C([0,T];H^m(\mathbb{R}^3)) \cap C^1([0,T];H^{m-1}(\mathbb{R}^3)).
$$
Furthermore, there exists a constant $C = C(m) >0$ such that
\begin{equation}\label{upp_u}
\sup_{0 \leq t \leq T} \| u^\mu (t) \|_{H^m} \leq (1+\| u^{\mu}_0 \|_{H^m})^{ \exp \left(CT(1+(1+\mu^{-1})(1+\mu) \| u^{\mu}_0 \|_{H^m}) \right)} =: C_{L},
\end{equation}
\end{proposition}
\begin{proof} We only give a brief sketch of the proof. One can obtain in a standard manner the following:
\begin{align*}
\frac 12 \frac {\ud}{\ud t} \| u^\mu \|_{\dot{H}^{m}}^2 &\leq C (1+\| \nabla \mathbf{v}_H \|_{BMO} \log^+ \| u^{\mu} \|_{H^m}) \| u^\mu \|_{\dot{H}^{m}}^2.
\end{align*}
Since Lemma~\ref{BMO} allows us to compute
\begin{align*}
\| \nabla \mathbf{v}_H \|_{BMO} \leq C (1+\mu^{-1}) \| \Delta_{\mu} \psi^{\mu} \|_{BMO}, 
\end{align*}
from a simple observation
\begin{align*}
\| \Delta_{\mu} \psi^{\mu} \|_{BMO} &\leq \| \Delta_\mu \psi^\mu \|_{L^\infty} = \| \Delta_\mu \psi^\mu_0 \|_{L^\infty} = \| \nabla_\mu \cdot u^{\mu}_0 \|_{L^\infty} \leq  C (1+\mu) \| u^{\mu}_0 \|_{H^m},
\end{align*}
it follows that
\begin{align*}
\| \nabla \mathbf{v}_H \|_{BMO} \leq C (1+\mu^{-1}) (1+\mu) \| u^{\mu}_0 \|_{H^m}.
\end{align*}
Combining the above estimates yields
\begin{align*}
 \frac{\ud}{\ud t}\| u^\mu \|_{H^{m}} &\leq C(1+(1+ \mu^{-1}) (1+\mu) \| u^{\mu}_0 \|_{H^m}  \log^+ \| u^{\mu} \|_{H^m})\| u^\mu \|_{H^{m}}.
\end{align*}
The Gr\"onwall inequality gives the desired result. 
\end{proof}

\begin{lemma}\label{BMO}
Let $S := \nabla_\mu \nabla_H^{\perp} (-\Delta_\mu)^{-1}$ for any $\mu \in \mathbb{R} \setminus \{ 0 \}$. Then there exists a constant $C > 0$ independent of $\mu$ such that
$$
\| S f \|_{BMO} \leq C \| f \|_{BMO}
$$
for any $f\in H^{s}(\bbR^3)$ with $s\geq \frac{3}{2}$.
\end{lemma}
\begin{proof} Without loss of generality, we assume that  $f\in\mathcal{S}(\bbR^3)$ where $\mathcal{S}$ denotes the Schwartz space. Once we prove the statement for such $f,$ then a simple density argument in $H^s$ will finish the proof. Note also that $\|f\|_{BMO} \leq C \|f\|_{H^{3/2}}.$

In Fourier variables, we observe that
$\mathscr{F}(S f)(\xi)= \frac{\eta_i \eta_j}{|\eta|^2}\mathscr{F}(f)\Big(\eta_1,\eta_2,\frac{\eta_3}{\mu}\Big)$ for $i,j = 1,2,3,$
once we set $(\eta_1,\eta_2,\eta_3) = (\xi_1,\xi_2, \mu\xi_3)$. Taking inverse Fourier transform in $\eta$, we get
$$
\mu (S f) (x_1,x_2,\mu x_3) = R_{i}R_{j} g(x), \,\,\mbox{and}\,\,\, g(x) = \mu f(x_1,x_2,\mu x_3),
$$
where $R_{i}$ and $R_{j}$ are the usual 3D Riesz transforms. For any cube $Q\subset \bbR^3,$ we define a $\mu$-normalized cube $Q_\mu$ by $Q_{\mu} = \{ y \in \bbR^3 : (y_1,y_2, \frac{y_3}{\mu}) \in Q \}.$ Then we can compute
\begin{align*}
\frac 1{|Q|} \int_Q | g(x) - g_Q| \mathrm{d}x = \frac 1{|Q|} \int_Q | \mu f(x_1,x_2,\mu x_3) - \mu f_{Q_\mu}| \mathrm{d}x = \frac \mu{|Q_\mu|} \int_{Q_\mu} | f(x) - f_{Q_\mu} | \mathrm{d}x,
\end{align*}
which yields, combined with the boundedness of Riesz transforms on $BMO$,
\begin{align*}
\| S f \|_{BMO} &= \| (S f) (x_1,x_2,\mu x_3) \|_{BMO} = \mu^{-1} \| R_{i}R_{j} g(x) \|_{BMO} \leq C \mu^{-1} \|  g(x) \|_{BMO} = C \| f \|_{BMO}.
\end{align*}
The proof is finished.
\end{proof}
\subsection{Continuity in $\mu$ for the QG equations}
The equivalence between the QG equations and the limit system \eqref{LS} has been shown in Section~\ref{GE_QG}; here, we consider \eqref{LS} instead of the QG system. We need to do some difference estimates for the two different rotation-stratification ratios $\mu>0$ and $\nu>0$. More specifically, we show that the limit system \eqref{LS}, and so \eqref{QG}, varies continuously in rotation-stratification ratio with respect to the $H^{m-1}$ norm for $m>\frac{5}{2}$, which is stated in Proposition \ref{conv_prop}. Then, by a simple addition/subtraction trick and a triangle inequality, we can finish the proof of Corollary~\ref{cor_conv}. We begin by the following difference lemma.


 
 
\begin{lemma}
Let $\mu, \nu \in (0,\infty)$ and $\mathbf{f} = (f_1,f_2,f_3,f_4) \in H^k$ for some $k \in \mathbb{N} \cup \{ 0 \}$. Then we have 
\begin{equation}\label{conv_est}
\| (P_\mu - P_\nu) \mathbf{f} \|_{H^k} \leq \frac {3|\mu - \nu|}{\sqrt{\mu \nu}} \| \mathbf{f} \|_{H^k}.
\end{equation}
\end{lemma}
\begin{proof}
Note that
\begin{equation}\label{muf_est}
\mathscr{F} P_\mu \mathbf{f} = \frac 1{|\xi_\mu|^2} \begin{pmatrix}
\xi_2^2 & -\xi_1 \xi_2 & 0 & -\mu \xi_2 \xi_3 \\
-\xi_1 \xi_2 & \xi_1^2 & 0 & \mu \xi_1 \xi_3 \\
0 & 0 & 0 & 0 \\
-\mu \xi_2 \xi_3 & \mu \xi_1 \xi_3 & 0 & \mu^2 \xi_3^2
\end{pmatrix} \mathscr{F} \mathbf{f}.
\end{equation}
Thus we have
$$
\mathscr{F} (P_\nu - P_\mu) \mathbf{f} = \frac {\mu - \nu}{|\xi_\mu|^2|\xi_\nu|^2} 
\begin{pmatrix}
(\mu + \nu) \xi_2^2 \xi_3^2 & -(\mu + \nu) \xi_1 \xi_2 \xi_3^2 & 0 & \xi_2 \xi_3 (|\xi_H|^2 - \mu \nu \xi_3^2) \\
-(\mu + \nu) \xi_1 \xi_2 \xi_3^2 & (\mu + \nu) \xi_1^2 \xi_3^2 & 0 & - \xi_1 \xi_3 (|\xi_H|^2 - \mu \nu \xi_3^2) \\
0 & 0 & 0 & 0 \\
\xi_2 \xi_3 (|\xi_H|^2 - \mu \nu \xi_3^2) & -\xi_1 \xi_3 (|\xi_H|^2 - \mu \nu \xi_3^2) & 0 & -(\mu + \nu)|\xi_H|^2 \xi_3^2
\end{pmatrix} \mathscr{F} \mathbf{f}.
$$
We compute as 
$$
|\sqrt{\mu \nu}(\mu + \nu) |\xi_H|^2 \xi_3^2| \leq (\mu^2 + \nu^2)|\xi_H|^2 \xi_3^2 \leq |\xi_\mu|^2|\xi_\nu|^2
$$
and see that
$$
|\sqrt{\mu \nu}\xi_1 \xi_3 (|\xi_H|^2 - \mu \nu \xi_3^2)| \leq \sqrt{\mu \nu} |\xi_H|^3 |\xi_3| + \sqrt{\mu \nu}^3 |\xi_H| |\xi_3|^3 \leq |\xi_H|^4 + \mu^2 \nu^2 |\xi_3|^4 \leq |\xi_\mu|^2|\xi_\nu|^2.
$$
By the Plancherel theorem, we obtain \eqref{conv_est}. This completes the proof.
\end{proof}



\begin{proposition}\label{conv_prop}
Let $\nu,\mu \in (0,\infty)$ with $|\nu - \mu| \leq \frac{\nu}{2}$ and $u_0 \in H^m(\bbR^3)$ for $m > \frac 52$. Let $u^\nu$ and $u^\mu$ be the corresponding global-in-time solutions to \eqref{LS}. Then, there exists a constant $C = C(\nu, m,T, \| u_0 \|_{H^m}) > 0$ such that
\begin{equation}\label{lim_conv_finite}
\sup_{t \in [0,T]} \| u^\nu - u^{\mu} \|_{H^{m-1}} \leq C|\nu - \mu|.
\end{equation}
\end{proposition}
\begin{proof}
Recalling the limit system \eqref{LS}, we can have
$$
\partial_t (u^\nu - u^\mu) + (P_\nu - P_\mu)[(u^\nu \cdot \widetilde{\nabla})u^\nu] + P_\mu [((u^\nu - u^\mu) \cdot \widetilde{\nabla})u^\nu] + P_\mu [(u^\mu \cdot \widetilde{\nabla})(u^\nu - u^\mu)] = 0.
$$
Note by \eqref{conv_est} that $\frac 12 \frac{\mathrm{d}}{\mathrm{d}t} \| u^\nu - u^\mu \|_{\dot{H}^{m-1}}^2$ is bounded above by
\begin{align*}
\frac {C|\nu - \mu|}{\sqrt{\mu \nu}} \| u^\nu \|_{H^m}^2 \| u^\nu - u^\mu \|_{H^{m-1}} &+ C \| u^\nu \|_{H^m} \| u^\nu - u^\mu \|_{H^{m-1}}^2 \\
&- \sum_{|\alpha|=m-1} \int \partial^{\alpha} (u^\mu \cdot \widetilde{\nabla})(u^\nu - u^\mu) \cdot \partial^{\alpha}  P_\mu(u^\nu - u^\mu) \mathrm{d}x.
\end{align*}
Using \eqref{conv_est} again, we have
\begin{equation*}
\begin{aligned}
&-\sum_{|\alpha|=m-1}\int \partial^{\alpha} (u^\mu \cdot \widetilde{\nabla})(u^\nu - u^\mu) \cdot \partial^{\alpha}  P_\mu(u^\nu - u^\mu) \mathrm{d}x \\
&\hphantom{\qquad\qquad}= -\sum_{|\alpha|=m-1} \int \partial^{\alpha} (u^\mu \cdot \widetilde{\nabla})(u^\nu - u^\mu) \cdot \partial^{\alpha} ((u^\nu - u^\mu) - ( P_\nu- P_\mu)u^\nu) \mathrm{d}x \\
&\hphantom{\qquad\qquad}\leq C \| u^\mu \|_{H^m} \| u^\nu - u^\mu \|_{H^{m-1}}^2 + \frac {C|\nu - \mu|}{\sqrt{\mu \nu}} \| u^\nu \|_{H^m} \| u^\mu \|_{H^m} \| u^\nu - u^\mu \|_{H^{m-1}},
\end{aligned}
\end{equation*}
and so the assumption $|\nu - \mu| \leq \nu/2$ yields
\begin{align*}
\frac{\mathrm{d}}{\mathrm{d}t} \| u^\nu - u^\mu \|_{H^{m-1}}\leq C C_L^2 \nu^{-1}|\nu - \mu| + C C_L \| u^\nu - u^\mu \|_{H^{m-1}},
\end{align*}
where the constant $C_{L}$ is defined in \eqref{upp_u}. Letting $y(t) = \| u^\nu - u^\mu \|_{H^{m-1}}$, we get
$$
\frac{\mathrm{d}}{\mathrm{d}t} y(t) \leq C C_L^2 \nu^{-1}|\nu - \mu| + C C_L y(t).
$$
The Gr\"{o}nwall inequality implies
\begin{align*}
y(t) &\leq \left(y(0) + C C_L^2 \nu^{-1} |\nu - \mu|t \right) e^{C C_L t}.
\end{align*}
Thus, together with $y(0) = \| (P_{\mu} - P_{\mu})u_0 \|_{H^{m-1}} \leq \frac{3|\mu-\nu|}{\sqrt{\mu\nu}} \| u_0 \|_{H^{m-1}}$, we obtain \eqref{lim_conv_finite}.
This completes the proof.
\end{proof}

\subsection{ Convergence for $\mu\neq 1$}

In the following theorem, we establish the convergence of \eqref{eq_main} to the QG equations when the rotation-stratification ratio $\mu\neq 1$ is fixed. This result is analogous to the one that was obtained in \cite{Tak1} for the one-scale singular limit $N\to\infty$ without consideration of Coriolis force. 

\begin{theorem}\label{thm1}
Let $m\in \mathbb{N}$ satisfy $m\geq 7$ and let  $4\leq q < \infty$. Fix any positive Burger number $\mu\neq 1$. For every $T > 0$ and for any initial data $u_0 \in H^m(\mathbb{R}^3)$ with $\widetilde{\nabla} \cdot u_0 = 0$, there exists a constant $R=R(m,q,\mu,T,\|u_0\|_{H^m})>0$ such that if $\sqrt{\Omega^2+N^2} > R$, then \eqref{eq_main} possesses a unique classical solution 
$$u \in C([0,T];H^m(\mathbb{R}^3)) \cap C^1([0,T];H^{m-1}(\mathbb{R}^3)).$$ Moreover, there exists a constant $C_\mu=C(m,q,\mu,T,\| u_0 \|_{H^m}) > 0$  such that
\begin{equation}\label{conv_1}
\| u - \nabla_{\mu} \psi^{\mu} \|_{L^q(0,T;W^{1,\infty})}^q \leq \frac {C_\mu} {\sqrt{\Omega^2+N^2}}
\end{equation}
as long as $\sqrt{\Omega^2+N^2} > R$. Here $\nb_\mu \psi^\mu$ is defined according to the notions in \eqref{QG}-\eqref{QG_def_initial}.
\end{theorem}

\begin{remark}\label{const}
 The constant $C_\mu$ appears in using the Littman theorem \cite{Littman} to prove the dispersive estimates. A careful look at Chapter VIII of \cite{Stein} tells us that, in one dimension, the size of $C_\mu$ blows up as $\mu\to1$. The exact size of $C_\mu$ in three dimensions is not known, but we are still able to obtain a lower bound for the blow-up rate, see Corollary~\ref{cor_div}. Geometrically, the Gaussian curvature of the surface $\{(\xi,p_{\mu}(\xi))\in\bbR^3\times\bbR:1/4\leq|\xi|\leq 4\}$ is zero at $\mu=1$ so that there is no proper ``oscillation" one can exploit to obtain the required decay.
\end{remark}

\begin{remark}
    For any fixed $\nu > 0$ with $\nu \neq 1$, one can see by a direct computation that $p_{\nu}(\xi)$ is $C^{\infty}$ on the set $D:=\{\xi \in \bbR^3 ; 1/4 \leq |\xi| \leq 4 \}$. Moreover, $\| p_{\nu} - p_{\mu} \|_{C^N(D)} \to 0$ as $\mu \to \nu$ for all $N \in \bbN \cup \{0\}$. Thanks to \cite[Lemma 3.3]{LT17}, there exist a positive constant $\varepsilon=\varepsilon(\nu) < \nu$ such that $C_\mu$ and $R$ are uniformly bounded above by some constants $\bar{C}$ and $\bar{R}$, respectively for all $\mu\in (\nu-\varepsilon,\nu+\varepsilon).$ This means that the convergence rate in \eqref{conv_1} can be made uniform near any fixed positive $\mu$ other than one.
\end{remark}


\begin{proof}[Proof of Theorem~\ref{thm1}]
This can be proved analogously to the proof of the main theorem in \cite{Tak1}. The required corresponding key ingredients are stated in Section~\ref{sec_Strichartz}-\ref{sec_modified}.    
\end{proof}

  The rotation-stratification ratio $\mu$ does not have to be entirely fixed during the convergence process. As long as $\mu$ converges to some $\nu\in(0,1)\cup(1,\infty)$ as $N\to\infty$ and $\Omega\to\infty$, there is some flexibility allowed for $\mu$. See below. 

\begin{corollary}\label{cor_conv} Assume the hypotheses of Theorem~\ref{thm1} on $q$, $m$, and $u_0$. Fix any $T>0$ and any $\nu\in(0,1)\cup(1,\infty)$. Let $\{(\Omega_k,N_k)\}_{k \in \mathbb{N}}$ be a sequence of pairs of positive numbers satisfying
\begin{equation*}
    N_k\to\infty,\quad \Omega_k \to \infty, \quad \mu_k=\frac{\Omega_k}{N_k} \to \nu \quad as \quad k\to\infty.
\end{equation*} Then the corresponding sequence $\{u^{(k)}\}_{k\in \bbN}$ of solutions to \eqref{eq_main} have the convergence property $$\| u^{(k)} - \nabla_{\nu} \psi^{\nu} \|_{L^q(0,T;W^{1,\infty})}^q \to 0 \qquad \mbox{as} \qquad k \to \infty.$$ Here $\nb_\nu \psi^\nu$ is defined according to the notions in \eqref{QG}-\eqref{QG_def_initial} with $\mu=\nu$.
\end{corollary}

\begin{remark}
The explicit convergence rate is given in \eqref{conv_3}.
\end{remark}

\begin{remark}
    In the proof of Corollary~\ref{cor_conv}, we implicitly used the fact that the constant $R$ in Theorem~\ref{thm1} becomes uniform in $\mu_k$ as long as $\mu_k$ is sufficiently close to $\nu$.
\end{remark}

\begin{proof}[Proof of Corollary \ref{cor_conv}]
We fix $\nu \in (0,1) \cup (1,\infty)$ and $(v_0,\theta_0) \in H^m(\bbR^3)$ for some $m \geq 7$ with $\nabla \cdot v_0 = 0$. Let a sequence $\{ (\Omega_k,N_k) \}$ satisfy $N_k \to \infty$ and $ \mu_k \to \nu$ as $k \to \infty$. For any $T>0$, by Theorem~\ref{thm1}, there is a constant $R>0$ such that the solution $$u^{(k)} \in C([0,T];H^m(\mathbb{R}^3)) \cap C^1(0,T;H^{m-1}(\mathbb{R}^3))$$ and \eqref{conv_1} are obtained, whenever $N_k > R$. From \eqref{u_phi} and Proposition~\ref{conv_prop}, we have $$\sup_{t \in [0,T]} \|\nabla_{\nu} \psi^{\nu} - \nabla_{\mu_k} \psi^{\mu_k} \|_{H^{m-1}} \leq C|\nu - \mu_k|.$$
By the triangle inequality and $H^{m-1}(\bbR^3) \hookrightarrow W^{1,\infty}(\bbR^3)$, it follows that $$\| u^{(k)} - \nabla_{\nu} \psi^{\nu} \|_{L^q(0,T;W^{1,\infty})}^q \leq 2^{q-1}\| u^{(k)} - \nabla_{\mu_k} \psi^{\mu_k} \|_{L^q(0,T;W^{1,\infty})}^q + 2^{q-1}\| \nabla_{\nu} \psi^{\nu} - \nabla_{\mu_k} \psi^{\mu_k} \|_{L^q(0,T;H^{m-1})}^q.$$ Therefore, we deduce 
\begin{equation}\label{conv_3}
    \| u^{(k)} - \nabla_{\nu} \psi^{\nu} \|_{L^q(0,T;W^{1,\infty})}^q \leq \frac{C_{\mu_k}}{N_k} + C|\nu - \mu_k|.
\end{equation}
This completes the proof.
\end{proof}

\subsection{Strichartz Estimates}\label{sec_Strichartz}
We will introduce a space-time Strichartz estimates for the linear propagator. Since $p_{\mu}(\xi)$ is homogeneous of degree zero (Section~\ref{sec_prelim}) it suffices to consider the operators
\begin{equation*}
   \begin{split}
       \mathcal{G}(t)\varphi(x) &:= \int_{\mathbb{R}^3}e^{ix\cdot \xi\pm it N  p_{\mu}(\xi)}\psi(\xi)^2\mathscr{F}\varphi(\xi)\,\ud\xi,\\
       \mathcal{H}(t)\varphi(x) &:=\int_{\mathbb{R}^3}e^{ix\cdot \xi\pm it N  p_{\mu}(\xi)}\psi(\xi)\mathscr{F}\varphi(\xi)\,\ud\xi,
   \end{split}
\end{equation*}
for $(t,x)\in\mathbb{R}\times\mathbb{R}^3$. Here we denote by $\psi$ the smooth real-valued cut-off function such that $\supp \psi \subset\{1/4\leq|\xi|\leq 4\}$ and $\psi(\xi)=1$ on $\{1/2 \leq |\xi|\leq 2\}.$ Then it is known that we have the following lemmas.
\begin{lemma}[\cite{IMT},\cite{LT17}]\label{supdecay}
Let $\mu\neq 1$. There exists a positive constant $C=C(\mu,\psi)>0$ such that
\begin{equation}\label{G_decay}
   \| \mathcal{G}(t) \varphi  \|_{L^{\infty}} \leq C(1+|Nt|)^{-\frac{1}{2}} \|\varphi\|_{L^1}
\end{equation}
for all $t\in\mathbb{R}$ and $\varphi\in L^1(\mathbb{R}^3).$ The same is true for $\mathcal{H}(t).$ The decay rate of \eqref{G_decay} is sharp. The information of $\Omega$ is encoded in the relation $\frac{\Omega}{N}=\mu$.
\end{lemma}

\begin{proof}
    We can compute
\begin{equation*}
   \operatorname{det}Hp_{\mu}(\xi)=(\mu^2-1)^3\frac{|\xi_H|^2\xi_3^4}{|\xi|^9|\xi_{\mu}|^3}
\end{equation*}
where we denote by $Hp_{\mu}(\xi)$ the Hessian matrix for $p_{\mu}(\xi).$ The rest of the proof is exactly the same with the proof of Lemma 3.2 in \cite{IMT} or equivalently Proposition 3.1 in \cite{LT17}.
\end{proof}

\subsection{Modified Linear Dispersive System}\label{sec_modified}
 Here we fix $u_0 \in H^{m+1}(\bbR^3)$ with $m \geq 2$ and introduce the modified linear system 
\begin{equation}\label{MLS}
\left\{
\begin{array}{ll}
\partial_t u^{\pm}\mp iN p_\mu(D)u^{\pm}+P_{\pm}(u^{\mu} \cdot\widetilde{\nabla})u^{\mu}=0,\\
\widetilde{\nabla}\cdot u^{\pm}=0,\\
u^{\pm}(0,x)=P_{\pm}u_0(x),
\end{array}
\right.
\end{equation}
where $u^\mu$ is a solution to \eqref{LS} such that $$
u^{\mu} \in C([0,T];H^{m+1}(\mathbb{R}^3)) \cap C^1([0,T];H^{m}(\mathbb{R}^3)).
$$ This can be viewed as an intermediate system between the linearized system \eqref{LEQ} and the full system \eqref{UEQ}. By the Duhamel principle, we have the following solution representation
\begin{equation}\label{duhamel_dispersive}
    u^{\pm}(t)=e^{\pm itNp_\mu(D)}P_{\pm}u_0 - \int_{0}^{t}e^{\pm i (t-\tau)N p_\mu(D)}P_{\pm}(u^{\mu}(\tau)\cdot\widetilde{\nabla})u^{\mu}(\tau)\,\ud\tau.
\end{equation}
Relying on the expression \eqref{duhamel_dispersive}, the linear estimates that were obtained in Section~\ref{sec_Strichartz} lead to the following lemma.
\begin{lemma}[\cite{Tak1}]
Fix $\mu\neq 1$. Let $m\in \mathbb{N}$ satisfy $m\geq 2.$  For each $u_0\in H^{m+1}$ with $\widetilde{\nabla}\cdot u_0=0,$ there exists a unique classical solution $u^{\pm}$ to \eqref{MLS} in the class
\begin{equation*}
    u^{\pm}\in C([0,T];H^m(\mathbb{R}^3))\cap C^{1}(0,T;H^{m-1}(\mathbb{R}^3)).
\end{equation*}
Furthermore, we obtain the uniform bound
\begin{equation}\label{MLSuniform}
    \sup_{t\in[0,T]}\|u^{\pm}(t)\|_{H^m}\leq \|u_0\|_{H^m}+C(m,T,\|u_0\|_{H^{m+1}}).
\end{equation}
For $q\in[4,\infty)$, there exist positive constants $C_1=C_1(q)$ and $C_2=C_2(m,q,T,\|u_0\|_{H^{m+1}})$ such that
\begin{equation}\label{dispersive}
    \|\nabla^{l}u^{\pm}\|_{L^q(0,T;L^\infty)}\leq C_1 N^{-\frac{1}{q}}(\|u_0\|_{H^{2+l}}+C_2)
\end{equation}
for $l=0,1,2,\ldots,m-2.$ The information of $\Omega$ is encoded in the relation $\frac{\Omega}{N}=\mu$.
\end{lemma}

\begin{proof}
    We use Lemma~\ref{supdecay}. Then we can follow \cite{Tak1} in an analogous fashion.
\end{proof}

\section*{Competing Interests}
 The authors have no competing interests to declare that are relevant to the content of this article.
\section*{Data Availability}
Data sharing not applicable to this article as no datasets were generated or analysed during the current study.
\section*{Acknowledgment}
\noindent The authors greatly thank the anonymous referees for their significant suggestions and corrections. The authors' gratitude also goes to Tai-Peng Tsai for his keen comments on important changes in the statements. J. Kim's work was supported by a KIAS Individual Grant (MG086501) at Korea Institute for Advanced Study. J. Lee's work was supported by National Research Foundations of Korea NRF-2021R1A2C1092830.
\bibliographystyle{amsplain}
\bibliography{Euler-Boussinesq}

\begin{thebibliography}{00}
	

\bibitem{Alex} A. Alexakis, A. van Kan, Energy cascades in rapidly rotating and stratified
turbulence within elongated domains. J. Fluid Mech. 933, A11 (2022)

\bibitem{BMN97} A. Babin, A. Mahalov, B. Nicolaenko, Regularity and integrability of 3D Euler and Navier–Stokes equations for rotating fluids. Asympt. Anal. 15, 103–150 (1997)

\bibitem{BMN99} A. Babin, A. Mahalov, B. Nicolaenko, Global regularity of 3D rotating Navier–Stokes equations for resonant domains. Indiana Univ. Math. J. 48, 1133–1176 (1999) 

\bibitem{BMNZ}
A. Babin, A. Mahalov, B. Nicolaenko, Y. Zhou,
On the asymptotic regimes and the strongly stratified limit of rotating Boussinesq equations.
Theoret. Comput. Fluid Dynamics 9, 223--251 (1997)

\bibitem{Baroud} C.N. Baroud, B.B. Plapp, H.L. Swinney, Z.-S. She, Scaling in three-dimensional and quasi-two- dimensional rotating turbulent flows. Phys. Fluids 15, 2091–2104 (2003)

\bibitem{BB}
A. Bourgeois, J. T. Beale,
Validity of the quasigeostrophic model for large scale flow in the atmosphere and ocean. SIAM J. Math. Anal. 25, 1023-1068 (1994)

\bibitem{Charney48} J. G. Charney, On the scale of atmospheric motions. Geofys. Publ. 17, 1-17 (1948)

\bibitem{Charney} J. G. Charney, Geostrophic turbulence. J. Atmos. Sci. 28, 1087-1095 (1971)

\bibitem{Charve05}
F. Charve, Global well-Posedness and asymptotics for a
geophysical fluid system. Comm. Partial Differential Equations 29:11-12, 1919-1940 (2005)



\bibitem{Chemin} J. Y . Chemin, A propos d'un probl\`eme de p\'enalisation de type antisym\'etrique. C.R.A.S. Paris Ser. I 321 (7), 861--864 (1995)
	
	\bibitem{CDGG02} J.Y. Chemin, B. Desjardins, I. Gallagher, E. Grenier, Anisotropy and dispersion in rotating fluids. Studies in Applied Mathematics, vol. 31, 171–192. North-Holland, Amsterdam (2002)
	
	\bibitem{CDGG06} J.Y. Chemin, B. Desjardins, I. Gallagher, E. Grenier, Mathematical geophysics: An introduction to rotating fluids and the Navier-Stokes equations. vol. 32, Oxford University Press, Oxford (2006)
	
	
	
	\bibitem{general1}
	B. Cheng, Q. Ju, S. Schochet,
	Three-scale singular limits of evolutionary PDEs.
	Arch. Rational Mech. Anal. 229, 601–625 (2018)
	
	\bibitem{rate2}
B. Cheng, Q. Ju, S. Schochet, Convergence rate estimates for the low Mach and low Alfvén number three-scale singular limit of ideal compressible magneto-hydrodynamics. M2AN Math. Model. Numer. Anal. 55 (suppl) S733–S759 (2021)
	

	
	\bibitem{EM} P. F. Embid, A. J. Majda, Averaging over fast gravity waves for geophysical flows with arbitrary potential vorticity. Comm. Partial Differential Equations 21, 619-658 (1996)
	
	
	\bibitem{rate1}
	J. Földes, S. Friedlander, N. Glatt-Holtz, G. Richards, Asymptotic analysis for randomly forced MHD. SIAM J. Math. Anal. 49 (6), 4440–4469 (2017)
	
	\bibitem{Gill}
	A. E. Gill, Atmosphere-Ocean Dynamics. Academic Press, New York (1982)
	

  \bibitem{Grenier}
  E. Grenier, Pseudo-differential energy estimates of singular perturbations. Comm. Pure Appl. Math. 50 (9), 821–865 (1997)

\bibitem{Hough}
	S. S. Hough, On the application of harmonic analysis to the dynamical theory of the tides. Part I : On Laplace's "Oscillations of the first species," and on the dynamics of
ocean currents. Phil. Trans. (A) 189 (IX), 201-258  (1897) 
	
	\bibitem{IMT} T. Iwabuchi, A. Mahalov, R. Takada, Global solutions for the incompressible rotating stably stratified fluids. Math. Nachr. 290 (4), 613-631 (2017)
	
	\bibitem{IT13} T. Iwabuchi, R. Takada, Global solutions for the Navier-Stokes equations in the rotational framework. Mathematische Annalen 357, 727-741 (2013)
	
	\bibitem{IT15} T. Iwabuchi, R. Takada, Dispersive effect of the Coriolis force and the local well-posedness for the Navier-Stokes equations in the rotational framework. Funkcialaj Ekvacioj 58, 365-385 (2015)
	
	\bibitem{KM}
	S. Klainerman, A. Majda,
	Singular limits of quasilinear hyperbolic systems with large parameters and the incompressible limit of compressible fluids. Comm. Pure Appl. Math. 43, 481-524 (1981)
	
	\bibitem{KLT14} Y. Koh, S. Lee, R. Takada, Strichartz estimates for the Euler equations in the rotational framework. Journal of Differential Equations 256, 707-744 (2014)
	
	\bibitem{KLT} Y. Koh, S. Lee, R. Takada, Dispersive estimates for the Navier-Stokes equations in the rotational framework. Advances in Differential Equations 19, 857-878 (2014)
	
	\bibitem{KT}
H. Kozono, Y. Taniuchi, Limiting case of the Sobolev inequality in BMO,
with application to the Euler equations. Comm. Math. Phys., 214, 191-200 (2000)
	
	\bibitem{LT17} S. Lee, R. Takada, Dispersive estimates for the stably stratified Boussinesq equations. Indiana Univ. Math. J. 66, 2037-2070 (2017)
	
	\bibitem{Littman} W. Littman, Fourier transforms of surface-carried measures and differentiability of surface averages. Bull. Amer. Math. Soc. 69, 766-770 (1963)
	
	


	\bibitem{MS}
	P. Mu, S. Schochet,
	Dispersive estimates for the inviscid rotating stratified Boussinesq equations in the stratification-dominant three-scale limit.
	J. Math. Pures Appl. 158, 90-119 (2022)

\bibitem{MW}
P. Mu, Z. Wei, Rotation-dominant three-scale limit of the Cauchy problem to the inviscid rotating stratified Boussinesq equations. J. Differ. Equ. 353, 385-419 (2023)
	
	\bibitem{Pedlosky}
	 J. Pedlosky, Geophysical Fluid Dynamics. Springer-Verlag, New York (1987)
	
	\bibitem{P16}  J. Proudman, On the motion of solids in a liquid possessing vorticity. Proceedings of the Royal Society of London. Series A 92,  408-424 (1916)
	
	\bibitem{shape}
J. N. Reinaud, D. G. Dritschel, C. R. Koudella,
The shape of vortices in quasi-geostrophic turbulence.
J. Fluid Mech. vol. 474, 175–192 (2003)

\bibitem{general2}
	S. Schochet, X. Xu, Moderately fast three-scale singular limits. SIAM J. Math. Anal. 52 (4), 3444–3462 (2020)
	
	\bibitem{SA18} K. Seshasayanan, A. Alexakis, Condensates in rotating turbulent flows. Journal of Fluid Mechanics 841,  434-462 (2018)
	
		\bibitem{Stein} E. M. Stein, Harmonic Analysis: Real-Variable Methods, Orthogonality, and Oscillatory Integrals. Princeton University Press, Princeton, NJ (1993)
	
	
	
	\bibitem{T17}  G.I. Taylor, Motion of solids in fluids when the flow is not irrotational. Proceedings of the Royal Society of London. Series A 93  99-113 (1917)
	
	\bibitem{Tak1} R. Takada, Strongly Stratified Limit for the 3D Inviscid Boussinesq Equations.
	Arch. Rat. Mech. Anal. 232, 1475-1503 (2019)
	
	\bibitem{Analogy} G. Veronis, The analogy between rotating and stratified flows. Ann. Rev. Fluid Mech. 2, 37-66 (1970)
	
	
	
	




	
	





	
 
	




	


\end{thebibliography}



\end{document}